\documentclass[12pt]{extarticle}
\usepackage{amsmath, amsthm, amssymb, hyperref, color}
\usepackage[shortlabels]{enumitem}
\usepackage{graphicx}
\usepackage[all]{xypic}
\usepackage{makecell}
\usepackage[final]{pdfpages}
\setboolean{@twoside}{false}
\usepackage{pdfpages}
\usepackage{caption}
\usepackage{subcaption}
\usepackage{scalefnt}
\usepackage{verbatim}
\usepackage{tikz}
\usepackage{forest}
\oddsidemargin \evensidemargin
\newtheorem{theorem}{Theorem}
\numberwithin{theorem}{section}
\newtheorem{proposition}[theorem]{Proposition}
\newtheorem{lemma}[theorem]{Lemma}
\newtheorem{corollary}[theorem]{Corollary}

\newtheorem{thmDef}[theorem]{Theorem/Definition}

\newtheorem{remark}[theorem]{Remark}
\newtheorem{example}[theorem]{Example}

\newcommand{\RR}{\mathbb{R}}

\newcommand{\PP}{\mathbb{P}}
\newcommand{\CC}{\mathbb{C}}
\newcommand{\ZZ}{\mathbb{Z}}
\newcommand{\CH}{\mathrm{CH}}
\newcommand{\Gr}{\mathrm{Gr}(1,\PP^3)}
\newcommand{\Sym}{\mathrm{Sym}}
\newcommand{\codim}{\mathrm{codim}\,}

\date{}
 
\title{\textbf{Changing Views on Curves and Surfaces}}

\author{Kathl\'en Kohn, Bernd Sturmfels and Matthew Trager}

\begin{document}

\maketitle

\begin{abstract}
\noindent Visual events in computer vision are studied from the perspective of
algebraic geometry. Given a sufficiently general curve or surface in
$3$-space, we consider the image or contour curve that arises by projecting
from a viewpoint. Qualitative changes in that curve occur when the viewpoint
crosses the visual event surface. We examine the components of this
ruled surface, and observe that these coincide with the iterated singular loci
of the coisotropic hypersurfaces associated with the original curve or surface. We
derive formulas, due to Salmon and Petitjean, for the degrees of these
surfaces, and show how to compute exact representations for all visual event
surfaces using algebraic methods.
\end{abstract}

\section{Introduction}
\label{sec1}

Consider a curve or surface in $3$-space, and pretend you are taking a picture
of that object with a camera. If the object is a curve, you see again a curve
in the image plane. For a surface, you see a region bounded by a curve, which
is called {\em image contour} or {\em outline curve}. The outline is the
natural sketch one might use to depict the surface, and is the projection of
the critical points where viewing lines are tangent to the surface. In both
cases, the image curve has singularities that arise from the projection,
even if the original curve or surface is smooth. Now, let your camera
travel along a path in $3$-space. This path naturally breaks up into
segments according to how the picture looks like. Within each segment, the
picture looks alike, meaning that the topology and singularities of the image
curve do not change.

The appearance of a solid object under a continuously varying viewpoint was
studied in the 1970s by Koenderink and van Doorn \cite{KD76}. Their motivation
came from visual perception in psychology and artificial intelligence.
Koenderink offers a detailed discussion in his remarkable book on {\em Solid
Shape} \cite{Koen}. On the mathematical side, the topic was studied in
singularity theory by Arnol'd and others~\cite{Arn,Ker,Plat}. In that setting,
the transitions between locally stable views are the non-generic singularities
from {\em catastrophe theory}. These catastrophes have been classified for
projection-generic surfaces. The catalogue consists of the following six {\em
visual events}. The first three names are due to Ren\'e Thom \cite{Thom}:

\begin{itemize}
\item[(L)]  Local events:   {\em lip},  {\em beak-to-beak},  and {\em
swallowtail}. \vspace{-0.1in}
\item[(M)] Multi-local events: {\em tangent crossing}, {\em cusp crossing},
and {\em triple point}.
\end{itemize}
In the 1980s, visual events became a research topic in computer vision
\cite{BD, Pae, PPK, Pon90}. Chapter~13 in the textbook by Forsyth and Ponce
\cite{ForPon} offers an introduction in that context.  The motivation in
computer vision was to give a description of all possible appearances of a
solid object using a finite number of stable views, or {\em aspects}. The
overall structure of aspects and events is encoded in the {\em aspect graph},
in which vertices correspond to aspects, and edges correspond to visual events
between stable views. Although these ideas never found much practical use,
several algorithms for computing aspect graphs of algebraic surfaces were
proposed. Test implementations involved both numerical and symbolic methods.
Ponce and Kriegman \cite{Pon90} and Rieger \cite{Rie92} studied the case of
orthographic projections of parametric algebraic surfaces. Methods for
implicit algebraic surfaces were introduced by~Petitjean {\it et
al}.~\cite{PPK} for orthographic projections, and by Rieger~\cite{Rie93} for
perspective projections. All examples shown in these articles are very special
low-degree surfaces. We here revisit this literature, now 25 years old, and
develop it further for today's applied algebraic geometry.

Our model for the object to be viewed is a smooth variety $X$ 
of dimension one or two in  complex projective space $\PP^3$. We assume that $X$ 
is defined over $\RR$ and the real locus $X_\RR$ is Zariski dense in $X$. Taking a
picture of $X_\RR$ is modeled by the linear projection $\pi:\PP^3
\dashrightarrow \PP^2$ with center $z$ (for {\em Zentrum}). This defines a
curve $C_z(X)$ in the image plane $\PP^2$. If $X$ is a curve, then $C_z(X)$ is
the closure of the image of $X$ under $\pi$. If $X$ is a surface, then
$C_z(X)$ is the branch locus of $\pi$ restricted to $X$. This is the closure
of the set of points in $\PP^2$ whose corresponding viewing lines are
tangent to $X$. Even though $X$ is smooth, the curve $C_z(X)$ has many singular
points. For a surface viewed from a general viewpoint $z$, the only
singularities in the contour are nodes and cusps. For a space curve, the image
curve has only nodes. As the center $z$ changes, the structure of its
singularities is locally constant. At some point, a transition occurs, and the
singularity structure changes. The {\em visual event surface} $\mathcal{V}(X)$
is the Zariski closure in $\PP^3$ of the set of these transition points. This
definition can be extended to singular curves and surfaces by excluding the
role of singular points on~$X$.

The visual event surface $\mathcal{V}(X)$ is usually reducible. If $X$ is a
general curve, then $\mathcal{V}(X)$ has three irreducible components. If $X$
is a general surface, then $\mathcal{V}(X)$ has five irreducible components.
These arise from the six events in (L) and (M) above. We shall explain
the geometry of these irreducible components, and we discuss how to compute
them in practice. An important caveat for applications is the distinction
between real and complex points. Algebraic methods  do not distinguish between
them. They apply to any complex curve or surface $X$ in  $\PP^3$. For any
particular $X$ that is defined over $\RR$, it can happen that some visual
events are not seen on its real points, i.e.,~they might live in the set $X \backslash
X_\RR$ of complex points.

The focus in this paper lies on curves and surfaces that are general in the
sense of algebraic geometry. Thus, for a surface $X$ in $\PP^3$ of degree $d$,
we assume that its equation is general among homogeneous polynomials of
degree $d$ in four variables. For a curve $X$ in $\PP^3$ of degree $d$ and
genus $g$, we assume that it is a general point in the Hilbert scheme of
such curves.

\smallskip

We now briefly describe the organization and main results in this article.
Section \ref{sec2} is devoted to ruled surfaces in $\PP^3$ and to its subclass
of developable surfaces. We introduce effective representations of ruled
surfaces, and we show how to compute with these. This is relevant because all
visual event surfaces are ruled. Their irreducible components are the ruled
surfaces in the bottom rows of Figures \ref{fig:curve-graph} and \ref{fig:surface-graph}. 
These arise as iterated singular loci of Chow and Hurwitz
threefolds in ${\rm Gr}(1,\PP^3)$, and of dual varieties in $(\PP^3)^*$.

In Section \ref{sec3} we develop the geometry of
visual event surfaces for curves in $\PP^3$. The three
irreducible components are the tangential surface, edge surface and trisecant surface.
These represent the three Reidemeister moves on the image curve, 
as shown in Figure \ref{fig:eins}. We present case studies that show
the computation of visual event surfaces for curves up to degree six.

Section \ref{sec4} concerns the visual event surface $\mathcal{V}(X)$ of a
general surface $X$ in $\PP^3$. The six events in (L) and (M) are depicted in
Figure \ref{fig:events}, which we discuss in detail. These events are
translated into the algebraic setting, where they correspond to the five
irreducible components of $\mathcal V(X)$. Their degrees are listed in Theorem
\ref{thm:surfaces}. These formulas were known classically: they appear in
paragraphs 597, 598, 599, 608 and 613 of Salmon's book \cite{Sal}. Modern
proofs were  given by  Petitjean~\cite{Pet}. In Section \ref{sec5} we present
new proofs, based on intersection theory in algebraic geometry, as seen in the
textbook by Eisenbud and Harris~\cite{EH}.

Section \ref{sec6} is devoted to practical methods for computing and
representing the visual events associated with a surface $X$ in $\PP^3$. This
is a non-trivial matter because the degrees of the ruled surfaces in the
output are very high, as seen in Table \ref{tab:surfaces}. For instance, if
$X$ is a quintic, then the degrees of the irreducible components of
$\mathcal{V}(X)$ range between $260$ and $930$.

\section{Ruled Surfaces and Developable Surfaces}

\label{sec2}

An irreducible surface in $\PP^3$ is ruled if it is covered by straight lines.
These lines are parameterized by some curve $C$, and they are known as the
{\em generators} of the surface. A first example are smooth quadratic surfaces
in $\PP^3$. These possess two rulings of lines over $\CC$. We refer to the
book by  Edge \cite{edge} for many classical results on ruled surfaces. In
this section we develop algebraic tools for computing and representing ruled
surfaces in practice.

Ruled surfaces arise naturally when taking pictures of an object in $3$-space.
We encounter them because all components of a visual event surface $\mathcal
V(X)$ are ruled. Indeed, every general point $z$ on $\mathcal{V}(X)$
determines a line of sight that has a special intersection with the curve or
surface $X$. Every point on the line shares this property with $z$ and hence
lies in $\mathcal{V}(X)$.

A key player is the Grassmannian of lines in $\PP^3$, here denoted  ${\rm Gr}(1,\PP^3)$. This is
a $4$-dimensional variety, embedded in $\PP^5$ via {\em Pl\"ucker
coordinates}. Every line is associated with a sextuple
$(p_{12}:p_{13}:p_{14}:p_{23}:p_{24}:p_{34})$ given by the $2 \times 2$-minors
of a $2 \times 4$-matrix whose kernel is the line. The Pl\"ucker coordinates satisfy the 
quadratic {\em Pl\"ucker relation}
\begin{equation}
\label{eq:pluckerrel}
p_{12} p_{34} - p_{13} p_{24} + p_{14} p_{23} \, = \,0.
\end{equation} 
The coordinate ring of ${\rm Gr}(1,\PP^3)$ is the polynomial ring in the six
unknowns $p_{ij}$ modulo the principal ideal generated by
(\ref{eq:pluckerrel}). We can also represent a line in $\PP^3$ by its {\em
dual Pl\"ucker coordinates} $(q_{12}: q_{13}: q_{14}: q_{23}: q_{24} :
q_{34})$. These are the $2 \times 2$-minors of a $2 \times 4$-matrix whose
rows are points that span the line. Primal and dual Pl\"ucker coordinates are
related by $q_{ij} = \sigma_{(ijkl)} p_{kl}$, where $i,j,k,l$ are distinct
indices, and $\sigma_{(ijkl)}$ is the sign of the permutation $(ijkl)$. 
Every line satisfying $p_{34} \not= 0$ has the parametric representation
\begin{equation}
\label{eq:lineparaone} 
z(t) \,\,= \,\,\bigl( \,-p_{34} \,:\, t p_{34} \,:\, p_{14}  - t p_{24} \,:\,  t p_{23}-p_{13} \,\bigr) .
\end{equation}
Consider now an irreducible curve $C$ in ${\rm Gr}(1,\PP^3)$ which has degree $d$ in $\PP^5$.
We write $I_C$ for its prime ideal in the coordinate ring.
The union of all lines on $C$ is a ruled surface $\mathcal{S}_C$ in~$\PP^3$.

\begin{lemma} \label{lem:everyruled}
The ruled surface $\mathcal{S}_C$ is irreducible and it has degree $d$ in $\PP^3$.
Conversely, every irreducible ruled surface in $\PP^3$ arises in this way from
some irreducible curve $C \subset {\rm Gr}(1,\PP^3)$.
\end{lemma}

\begin{proof}
This is one of the basic facts derived in Edge's book \cite[Chapter I, \S 26]{edge}.
\end{proof}

The defining polynomial of the surface $\mathcal{S}_C \subset \PP^3$
 can be computed from the equations
 \begin{equation}
\label{eq:skewsymP}
\begin{pmatrix}
0  & p_{12} &  p_{13} & p_{14} \,\\
-p_{12} & 0 & p_{23} & p_{24} \,\\
-p_{13} & -p_{23} & 0 & p_{34} \,\\
-p_{14} & -p_{24} & -p_{34} & 0
\end{pmatrix} \cdot \begin{pmatrix}  x_1 \\  x_2 \\ x_3 \\ x_4 \end{pmatrix} 
\,\, = \,\,
\begin{pmatrix} 0 \\ 0 \\ 0 \\ 0 \end{pmatrix}.
\end{equation}
We add these four bilinear forms to the ideal $I_C$, and then we saturate
 with respect to the irrelevant ideal $\langle p_{12}, p_{13}, p_{14},
p_{23}, p_{24},p_{34} \rangle$ of $\PP^5$.
The resulting ideal is prime, and it describes the incidence correspondence of
points on lines that are in the curve $C$. Now, by eliminating the
unknowns $p_{ij}$, we obtain a principal homogeneous prime ideal in 
$\RR[x_1,x_2,x_3,x_4]$. The generator of this 
ideal is the polynomial of degree $d$ that defines the desired surface.

This computation can be reversed. Given a surface $\mathcal{S}$ in $\PP^3$,
we can compute the {\em Fano scheme} of all lines on $\mathcal{S}$. This lives in
${\rm Gr}(1,\PP^3)$. Its ideal in  $\RR[p_{12},p_{13},p_{14},p_{23},p_{24},p_{34}]$
is obtained by substituting (\ref{eq:lineparaone})
into the equation of $\mathcal{S}$, extracting the coefficients of the resulting polynomial in $t$,
and saturating their ideal  by $\langle p_{34} \rangle$.
The Fano scheme  is usually empty or consists of points.  However, if it
 is a curve $C$, then the surface is ruled and $\mathcal{S} = \mathcal{S}_C$.

The {\em dual projective space} $(\PP^3)^*$ consists of all planes in $\PP^3$.
Homogeneous coordinates on $(\PP^3)^*$ are denoted $(y_1 : y_2 : y_3 :
y_4)$. This point $y$ represents the plane $\{ x \in \PP^3: \sum_{i=1}^4 y_i
x_i = 0\}$. Given an arbitrary subvariety ${X} $ in $ \PP^3 $, its {\em dual}
${X}^\vee \subset (\PP^3)^*$ is the Zariski closure of the set of all planes
$y$ that are tangent at a smooth point of ${X}$. Typically, the dual
$\mathcal{S}^\vee$ of a surface of degree at least two will be a surface in
$(\PP^3)^*$. However, it can happen that $\mathcal{S}^\vee$ is a curve. In
that case,  $\mathcal{S}$ is a {\em developable surface}. Each
developable surface $\mathcal{S}$ is encoded by its dual curve
$\mathcal{S}^\vee$ since we can recover the surface by the biduality relation
$\,\mathcal{S} = (\mathcal{S}^\vee)^\vee$.

\begin{theorem}
\label{thm:everydev} Every developable surface $\mathcal{S}$ is a ruled
surface, i.e.,~it satisfies $\mathcal{S} = \mathcal{S}_C$ for some curve $C$
in ${\rm Gr}(1,\PP^3)$. For a curve $C$ in ${\rm Gr}(1,\PP^3)$, the
corresponding ruled surface $\mathcal{S}_C$ is developable if and only if all
tangent lines of $\,C$ in $\PP^5$ are  contained in the Pl\"ucker quadric.
\end{theorem}

\begin{proof}
For the first statement see
\cite[Chapter V, \S 344]{edge}.
The second is~\cite[Prop.~12.4.1]{BCMG}.
\end{proof}

A developable surface $\mathcal{S} = \mathcal{S}_C$ that is not a cone has
three distinct encodings as a curve. First, there is the
curve $C$ in the Grassmannian ${\rm Gr}(1,\PP^3)$. Second,
there is the dual curve $\mathcal{S}^\vee$ in  $(\PP^3)^*$.
We saw how to recover $\mathcal{S}$ from these encodings.
Finally, there is the {\em edge of regression} $E(\mathcal{S})$ 
which lies on the surface $\mathcal{S}$ in $\PP^3$. 
Points in $E(\mathcal{S})$ are
planes in $(\PP^3)^*$ that intersect the curve $\mathcal{S}^\vee$ with multiplicity three.
The surface $\mathcal{S}$ is the tangential surface of $E(\mathcal{S})$,
i.e.,~it is the union of lines that are tangent to $E(\mathcal{S})$ (see~\cite[page 111]{Pie}). This also verifies
that $\mathcal{S}$ is indeed a ruled surface. 
The curves $E(\mathcal{S})$ and $\mathcal{S}^\vee$ are also related by a biduality relation; namely, $\mathcal{S}^\vee$
is the set of planes in $\PP^3$ that intersect $E(\mathcal{S})$ with multiplicity three~\cite[Thm.~5.1]{Pie1}.
Moreover, the tangent lines of $E(\mathcal{S})$ and $\mathcal{S}^\vee$ are dual to each other.
This situation degenerates when the surface $\mathcal{S}$ is a cone, which means that 
its dual $\mathcal{S}^\vee$ is a plane curve.
In that special case, the edge of regression $E(\mathcal{S})$ is the vertex of the cone $\mathcal{S}$.

We illustrate the three curve encodings of a developable surface with a simple
example.

\begin{example} \rm
Let $\mathcal{S}$ be  the surface of degree six in $\PP^3$ that is defined by the polynomial
\begin{equation*}
 f \,\,= \,\,
16 x_2^3 x_3^3-27 x_1^2 x_3^4+6 x_1 x_2^2 x_3^2 x_4-27 x_2^4 x_4^2+48 x_1^2 x_2 x_3 x_4^2-16 x_1^3 x_4^3.
\end{equation*}
This is the surface in \cite[\S 3, eqn.~(9)]{SS}. We verify that $\mathcal{S}$ is developable
by computing the ideal of its dual variety $\mathcal{S}^\vee \subset (\PP^3)^*$.
This shows that $\mathcal{S}^\vee$ is a smooth rational quartic curve:
\begin{equation}
\label{eq:annacurve1} \bigl\langle \,
    y_2 y_3-4 y_1 y_4\,,\,y_3^3+4 y_2 y_4^2 \,,\, y_1 y_3^2+y_2^2 y_4 \,,\,y_2^3+4y_1^2 y_3 \,\bigr\rangle.  
\end{equation}
The curve $C$ in the Grassmannian ${\rm Gr}(1,\PP^3)$ that encodes
the ruling of $\mathcal{S} = \mathcal{S}_C$ has the ideal
\begin{equation*}
  I_C \,\,\, = \,\,\,
\bigl\langle\,
2 p_{14}-p_{23} \, , \,
p_{24}^2+3p_{13} p_{34} \, , \,
p_{13} p_{24}-9 p_{12} p_{34} \, , \,
p_{23}^2-16 p_{12}  p_{34} \, , \,
p_{13}^2+3 p_{12} p_{24} \,
\bigr\rangle .
\end{equation*}
This ideal defines the Fano scheme of $\mathcal{S}$ in $\PP^5$.
Finally, the edge of regression $E(\mathcal{S})$ is the rational quartic curve
$\,\{(s^4:s^3t: st^3: t^4)\}\,$ in $\PP^3$. The ideal of this curve equals
\begin{equation}
\label{eq:annacurve2} \bigl\langle \,
    x_2 x_3-x_1 x_4\,,\,x_3^3  - x_2 x_4^2 \,,\, x_1 x_3^2 - x_2^2 x_4 \,,\,x_2^3 -  x_1^2 x_3 \,\bigr\rangle.  
\end{equation}
This curve has $\mathcal{S} $ as its tangential surface. Note that (\ref{eq:annacurve1}) is isomorphic to
(\ref{eq:annacurve2}). This reflects the isomorphism between (9) and (10) in
\cite{SS}. All of these computations can be reversed. This example shows how various 
objects can serve as a representation of the surface~$\mathcal S$.
\hfill $\diamondsuit$
\end{example}

Many of the ruled surfaces $\mathcal{S}_C$ we shall encounter in later sections have the property
that their defining polynomial $f$ is extremely large and impossible to compute symbolically.
In such cases, the curve $C $ in ${\rm Gr}(1,\PP^3) \subset \PP^5$ is more manageable, and we can often
compute generators for its ideal $I_C$.  This encoding of the ruling
enables us to carry out computations with the surface  $\mathcal{S}_C$.
For example, suppose $ \mathcal{S}_C$ has degree $d$ and consider
 a general line $L$ in $\PP^3$. We may wish to compute the $d$  points in the
 intersection $\mathcal{S}_C \cap L$. This problem arises in our computer vision application
 when the camera travels along $L$. The real intersection points with the visual event surfaces
 are precisely the visual events we are interested in.

Fix two points  $(a_1:a_2:a_3:a_4)$ and $(b_1:b_2:b_3:b_4)$ on $L$, and parameterize $L$ by
\begin{equation}
\label{eq:linepara} x_i \,=\, s a_i + t b_i \quad {\rm for} \,\,\,i=1,2,3,4. 
\end{equation}
To compute $\mathcal{S}_C \cap L$ from $I_C$, we substitute 
(\ref{eq:linepara}) into (\ref{eq:skewsymP}), we add the resulting
four bilinear forms to $I_C$, we saturate with respect to
$\langle p_{12},\ldots,p_{34} \rangle$, and we then eliminate
the six Pl\"ucker coordinates. The result is the principal
ideal in $\mathbb{R}[s,t]$ that is generated by the binary form
\begin{equation}
\label{eq:binfor}   f \bigl(\, s a_1 + t b_1, \,
s a_2 + t b_2, \,s a_3 + t b_3, \,s a_4 + t b_4 \,\bigr). 
\end{equation}
Thus, even when $f$ is unknown, we can compute its specialization
(\ref{eq:binfor}) directly from $I_C$.

When $\mathcal S$ is developable, the specialization \eqref{eq:binfor} can
also be obtained from the ideal $I(\mathcal S^\vee)$. Let $J$ be a
Jacobian matrix for the ideal $I(\mathcal{S}^\vee)$ in $\RR[y_1,y_2,y_3,y_4]$.
This matrix has four columns. Let $J_x$ be the matrix obtained from $J$ by
adding one more row, namely the vector  $(x_1,x_2,x_3,x_4)$ in
(\ref{eq:linepara}). We now add the $3 \times 3$-minors of $J_x$ to the ideal
$I(\mathcal{S}^\vee)$, we saturate with respect to the ideal of $2 \times
2$-minors of $J$, and then we eliminate the unknowns $y_1,y_2,y_3,y_4$. The
result is the desired principal ideal  (\ref{eq:binfor}) in $\RR[s,t]$. See
Example~\ref{ex:stabbing_victory} for an application.

These strategies can be adapted to compute the plane curve that is obtained as
the intersection of a ruled or developable surface $\mathcal S$ with a fixed
plane $H$ in $\PP^3$. For event surfaces, this corresponds to restricting the
camera movement to a plane, or to assuming that all projections are
orthographic (which means that the viewpoint lies on the plane at infinity).
It is sufficient to parameterize the points on $H$ by writing $x_i = s
a_i + t b_i + u c_i$ in~\eqref{eq:binfor}.

\smallskip

\paragraph*{Associated ruled surfaces.}

The ruled surfaces of interest to us arise
from an arbitrary curve or surface $X$ in $\PP^3$. 
They represent families of planes and lines that intersect $X$ with prescribed multiplicities and are shown in the bottom rows of Figures \ref{fig:curve-graph}  and \ref{fig:surface-graph}.
For a general curve or surface $X$ in $\PP^3$, the rows of these diagrams correspond to codimension in $(\PP^3)^*$ or ${\rm Gr}(1,\PP^3)$.
The shown subvarieties consist of lines and planes that intersect $X$ with various multiplicities $m$.
A solid edge from $Y_1$ to $Y_2$ means that $Y_2$ is an irreducible component of the
singular locus of $Y_1$. A dashed edge just means that $Y_2$ is 
contained in $Y_1$.
Below the ambient spaces $(\PP^3)^*$ and ${\rm Gr}(1,\PP^3)$
we see the {\em coisotropic hypersurfaces} studied by
 Gel'fand, Kapranov and Zelevinsky in \cite[Sec.~4.3.B]{GKZ}.
See also \cite{Kohn, Stu}. These codimension one loci~are:
\begin{itemize}
\item the dual surface $X^\vee $ in $(\PP^3)^*$, \vspace{-0.12in}
\item the \emph{Chow threefold} ${\rm Ch}(X)$ in ${\rm Gr}(1,\PP^3)$, consisting of lines that meet the curve $X$, \vspace{-0.12in}
\item the \emph{Hurwitz threefold} ${\rm Hur}(X)$ in ${\rm Gr}(1,\PP^3)$, of lines that are tangent to the  surface $X$. 
\end{itemize}
Every irreducible hypersurface in ${\rm Gr}(1,\PP^3)$ is defined by one equation in Pl\"ucker coordinates,
 which is unique up to scaling and modulo the Pl\"ucker quadric.
In the two cases above, this equation is called \emph{Chow form} resp.~\emph{Hurwitz form} of $X$.
As we shall see in Sections \ref{sec3} and \ref{sec4}, the irreducible components of
the visual event surface of $X$ are (iterated) singular loci of the coisotropic hypersurfaces associated to $X$. 
The developable components are dual to the singular curves in the dual surface $X^\vee$.
The non-developable components are parameterized by the singular curves in the singular locus of the Chow or Hurwitz threefold of~$X$.

We first consider a general smooth curve $X$ in $\PP^3$. 
The left diagram in Figure~\ref{fig:curve-graph} depicts the landscape in $(\PP^3)^*$.  The dual surface $X^\vee$ 
consists of planes that meet $X$ with multiplicity $2$. The singular locus of $X^\vee$
is the union of two irreducible curves, whose
points are osculating planes $(m=3)$ and bitangent planes $(m=2+2)$.
The symbols that denote our loci, like $\mathcal{T}^p(X)$ and $\mathcal{E}^p(X)$,
 will be explained in Sections \ref{sec3} and \ref{sec4}.
The right diagram in Figure \ref{fig:curve-graph} shows
the landscape in the Grassmannian ${\rm Gr}(1,\PP^3)$.
We refer to \cite[Thm.~1.1 and Sec.~7]{KNT} for precise statements and proofs, 
also for the right diagram in Figure \ref{fig:surface-graph}.
The singular locus of the Chow threefold ${\rm Ch}(X)$ is the surface ${\rm Sec}(X)$ in ${\rm Gr}(1,\PP^3)$ of secant lines, i.e., lines that meet $X$ twice.
The singular locus of  ${\rm Sec}(X)$ is the curve $\mathcal{D}^\ell(X)$ of trisecant lines.
The curve $\mathcal{T}^\ell(X)$ of tangent lines is contained in ${\rm Sec}(X)$
but it does not belong to the singular~locus.

\begin{figure}[htbp]
  \centering
  \begin{tikzpicture}[sibling distance=10em,
    every node/.style = {
    align=center}]]
    \node {$(\PP^3)^*$\\{\footnotesize $m=1$}}
    child[dashed] {node[solid] {$X^\vee$\\{\footnotesize $m=2$}}
    child[solid] { node {$\mathcal T^p(X)$\\{\footnotesize $m=3$} } }
    child[solid] { node {$\mathcal E^p(X)$\\{\footnotesize $m=2+2$}}}};
    \end{tikzpicture}
    \qquad \qquad
    \begin{tikzpicture}[sibling distance=10em,
    every node/.style = {
    align=center}]]
    \node {${\rm Gr}(1,\PP^3)$\\\footnotesize $m=0$}
    child[dashed] {node {${\rm Ch}(X)$\\ \footnotesize $m=1$}
    child[solid] {node {${\rm Sec}(X)$\\ \footnotesize $m=1+1$}
    child[solid] {node {$\mathcal D^\ell(X)$\\ \footnotesize $m=1+1+1$}}
    child[dashed]{node {$\mathcal T^\ell(X)$\\ \footnotesize $m=2$}}}};
    \end{tikzpicture}
  \caption{Loci of planes and lines that meet a curve $X$ with assigned multiplicities.}
  \label{fig:curve-graph}
\end{figure}
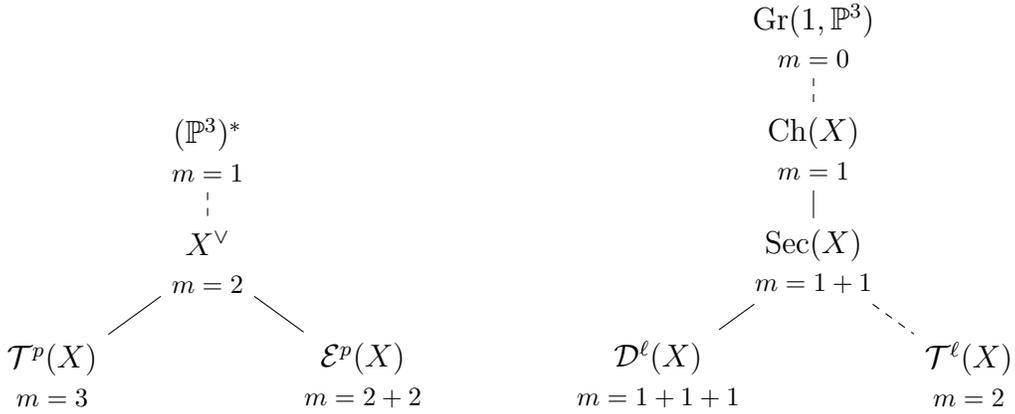

In Figure~\ref{fig:surface-graph}, we consider various loci associated with a general smooth surface $X$ in
$\PP^3$. The dual surface $X^\vee$ is singular along two irreducible curves.
The nodal component $\mathcal{E}^p(X)$ of its singular locus is the set of all bitangent planes, and
the cuspidal component $\mathcal{P}^p(X)$ is the set of all planes that
intersect $X$ with multiplicity three at a point. The Hurwitz threefold ${\rm
Hur}(X)$ is singular along two irreducible surfaces. Its nodal component ${\rm
Bit}(X)$ contains all bitangent lines, and its cuspidal component ${\rm PT}(X)$
comprises all principal tangents, i.e., lines that meet $X$ with
multiplicity three. The latter was denoted ${\rm Infl}(X)$ in \cite{KNT}. These
surfaces contain three special curves $\mathcal{F}^\ell(X)$,
$\mathcal{C}^\ell(X)$ and $\mathcal{T}^\ell(X)$, indicating lines that meet
$X$ with multiplicity $4$, or $3{+}2$, or $2{+}2{+}2$. For instance,
$2{+}2{+}2$ refers to  tritangent lines. Sections \ref{sec4}, \ref{sec5} and
\ref{sec6} are devoted to the ruled surfaces in $\PP^3$ that are represented by these curves.

\begin{figure}[htbp]
  \centering
  \begin{tikzpicture}[sibling distance=10em,
    every node/.style = {
    align=center}]]
    \node {$(\PP^3)^*$\\{\footnotesize $m=1$}}
    child[dashed] {node[solid] {$X^\vee$\\ \footnotesize $m=2$}
    child[solid] { node {$\mathcal P^p(X)$\\ \footnotesize $m=3$} }
    child[solid] { node {$\mathcal E^p(X)$\\ \footnotesize $m=2+2$} }};
    \end{tikzpicture}
    \qquad \qquad
    \begin{forest}
    [{${\rm Gr}(1,\PP^3)$ \\ \footnotesize $m=1$}, align=center, base=bottom
      [{${\rm Hur}(X)$\\ \footnotesize $m=2$}, align=center, base=bottom, edge=dashed
        [{${\rm PT}(X)$\\\footnotesize $m=3$}, align =center, base = bottom, name = PTnode
          [{$\mathcal F^\ell(X)$\\\footnotesize $m=4$}, align=center, base=bottom, name = Fnode]
          [, no edge ]
        ] 
        [{${\rm Bit}(X)$\\\footnotesize $m=2+2$}, align=center, base=bottom, name = BITnode
          [{$\mathcal C^\ell(X)$\\\footnotesize $m=3+2$}, align=center, base=bottom, name = Cnode]
          [, no edge ]
          [{${\mathcal T}^\ell(X)$\\\footnotesize $m=2+2+2$}, align=center, base=bottom]
        ]
      ]
    ]
    \draw (PTnode.south east)[dashed] -- (Cnode.north west);
    \draw (BITnode.south west)[dashed] -- (Fnode.north east);
    \end{forest}
  \caption{Loci of planes and lines that meet a surface $X$ with assigned multiplicities.}
  \label{fig:surface-graph}
\end{figure}
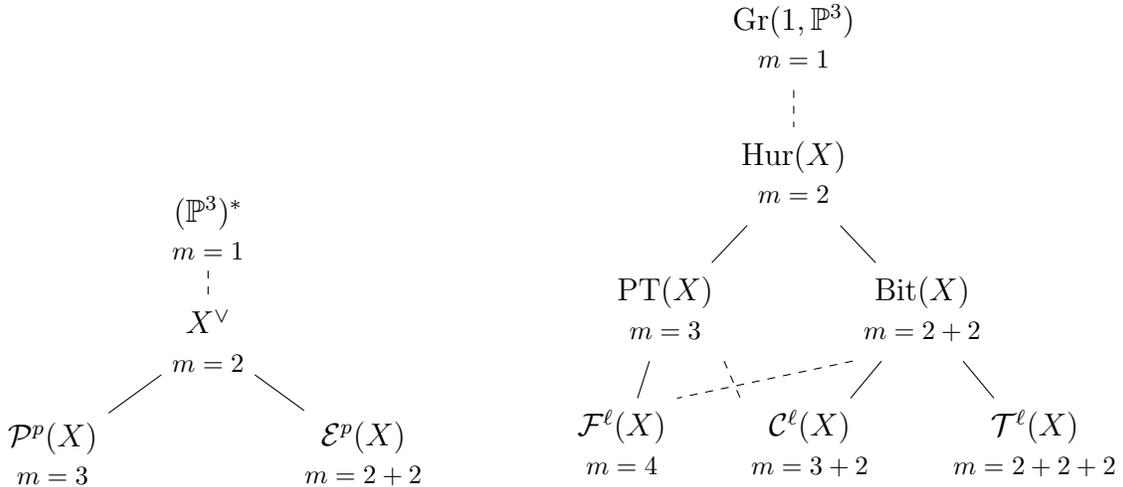

For each of the five curves at the bottom of Figure~\ref{fig:surface-graph},
there is also an associated curve on $X$. It consists of the points on $X$
where the special intersection occurs. For example, the curve associated with
$\mathcal E^{p}(X)$ is the locus of points on $X$ that lie on bitangent
planes. These are the contact points on a curved object when it is rolled on a
table. Our favorite terminology for this curve is due to Cayley: he calls it
the \emph{node-couple curve}. For $\mathcal P^{p}(X)$ and $\mathcal
F^{\ell}(X)$, the special contact occurs at a single point, and we can give a
more detailed description. At a general point $x$, the surface has two
principal tangents. These are the tangent lines of the nodal curve obtained by
intersecting $X$ with its tangent plane at $x$. The same lines are the
intersection of the tangent plane with the Hessian quadric at $x$ defined by
\begin{equation*}
 \quad
y \cdot H_f(x) \cdot y^T \,\, = \,\, 0, \quad
\hbox{where $y =  (y_1 : y_2 : y_3 : y_4)\,\,$ and} \quad
 H_f \,=\, \left(\frac {\partial^2 f}{\partial x_{i} \partial x_{j}}\right).
\end{equation*}

Exceptional situations occur at {\em flecnodal} and {\em parabolic} points
$x$. At a flecnodal point, one of the two principal tangents has intersection
multiplicity four. Such a line is called a {\em flecnodal line}. At a parabolic
point, the Hessian matrix $H_f(x)$ drops rank, and the two principal tangents
degenerate to a double line. At these points, the intersection of $X$ with its
tangent plane has a cusp at $x$. The locus of all parabolic points is the
curve given by the intersection of $X$ with the {\em Hessian surface}
$\{\det(H_f)=0\}$. Over the real numbers, the parabolic curve is the boundary
between the {\em elliptic} and {\em hyperbolic} regions on $X$, where the two
principal tangents are respectively both complex or both real. 

The curve
$\mathcal P^p(X)$ is the set of tangent planes at parabolic points, and the
curve $\mathcal F^{\ell}(X)$ is the set of flecnodal lines at flecnodal
points. The parabolic and flecnodal curves always intersect tangentially, at
special points known as {\em godrons} (or {\em cusps of the Gauss map}).
Interestingly, the node-couple curve also passes through the godrons, and has
the same tangent as the parabolic and flecnodal curves~\cite[p.~170]{cayley}.

\begin{remark} \rm A formal study of the singular loci of the families of
lines and planes described in this section presents many technical
challenges. This is the topic of \cite{KNT}. For example, in the course of examining parabolic surfaces, we
discovered a small error in \cite[Section 4]{ABT}, where Arrondo {\it et al.}
consider the incidence variety for principal tangents of $X$: $$ Y_2 \,=
\,\bigl\{ \,(x,L) \in X \times {\rm Gr}(1,\PP^3) \,: \,
\hbox{$L$ intersects $X$ at $x$ with multiplicity at least $3$} \,\bigr\}. $$
Lemma 4.1 b) in \cite{ABT} states that the surface $Y_2$ is singular at points
$(x,L)$ for which $x$ is parabolic. This is incorrect. A general cubic surface
$X$ has a degree $12$ curve of parabolic points. However, the incidence
variety $Y_2$ is smooth. This is shown by direct computation.
\end{remark}

\section{Views of Curves}
\label{sec3}

In this section we study the visual events for a general
curve $X$ in $\PP^3$ of degree $d$ and genus $g$. In particular, $X$ is smooth and
irreducible. The three visual events correspond to the three {\em Reidemeister
moves} that are familiar from knot theory. They are shown in
Figure~\ref{fig:eins}.

\begin{figure}[h]
  \begin{center}
\includegraphics[height=6.5cm]{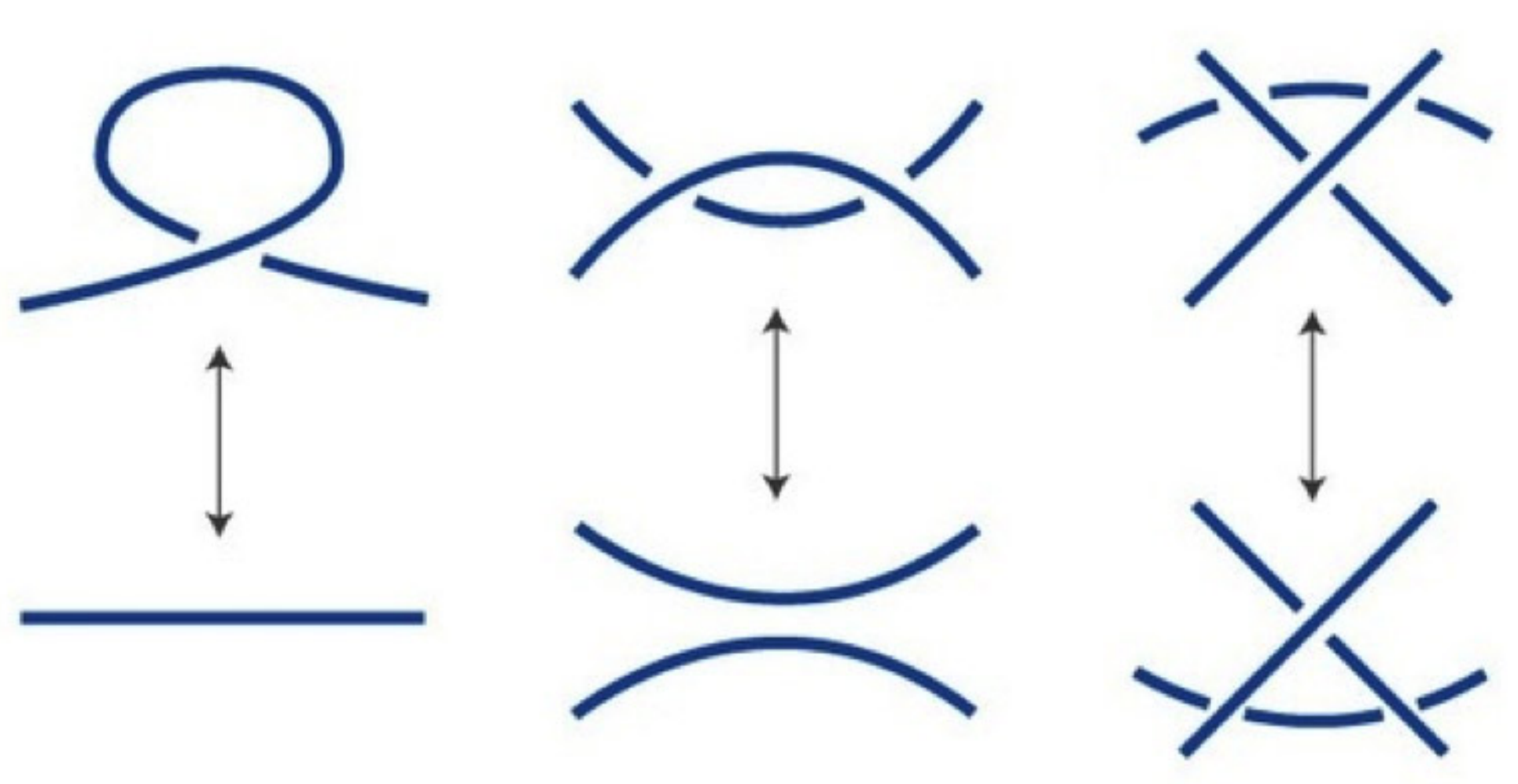}
\vspace{-0.2in}
\end{center}    \caption{\label{fig:eins} 
Changing views of a curve correspond to Reidemeister moves.
The viewpoint $z$ crosses the tangential surface 
(left), edge surface (middle), or trisecant surface (right).  }
 \end{figure}

\smallskip

The three components of the visual event surface of a space curve $X$ are as
follows:
\begin{enumerate}
\item The {\em tangential surface} $\mathcal{T}(X)$, also known as the {\em
tangent developable}, is the union of all tangent lines to $X$. It represents
viewpoints $z$ such that the plane curve $C_z(X)$ has a cusp. When $z$ crosses
$\mathcal{T}(X)$, a node on $C_z(X)$ transitions from being real to complex.
\item The {\em edge surface} $\mathcal{E}(X)$ is the union of all secant lines
that are edges. An {\em edge} is the line spanned by two points on $X$ whose
tangent lines lie in a common plane. This surface represents viewpoints
$z$ such that the plane curve $C_z(X)$ has a tacnode. When $z$ crosses $\mathcal E(X)$, a pair
of nodes transitions between being real and  being complex.
\item The {\em trisecant surface} $\mathcal{D}(X)$ is the union of all lines that are
spanned by triples of collinear points on $X$ (the symbol $\mathcal{D}$ stands
for {\em drei}). This represents viewpoints $z$ such that $C_z(X)$ has a
triple point. When $z$ crosses $\mathcal{D}(X)$, the real curve $C_z(X)$
experiences a triangle crossing, but the real singularity structure is
unchanged.
\end{enumerate}

The following classical theorem characterizes the expected degrees of these ruled surfaces.

\begin{theorem}
\label{thm:curves}
For a general space curve $X$ of degree $d$ and genus $g$, 
the degrees of the tangential surface $\mathcal{T}(X)$,
 the edge surface $\mathcal{E}(X)$
and the trisecant surface $\mathcal{D}(X)$ are as follows:
$$
\begin{matrix}
{\rm deg}\bigl(\mathcal{T}(X) \bigr) & = & 2(d+g-1), \\
{\rm deg}\bigl(\mathcal{E}(X) \bigr)& =  & 2(d-3)(d+g-1), \\
 {\rm deg}\bigl(\mathcal{D}(X) \bigr) & = & \frac{(d-1)(d-2)(d-3)}{3}-(d-2)g. \\
 \end{matrix}
$$
\end{theorem}

\begin{proof} The degree of the tangential surface $\mathcal T(X)$ is the Riemann-Hurwitz number
$2d+2g-2$. This coincides with the degree of the dual surface $X^\vee$. See
\cite[page 111]{Pie} for a geometric derivation and \cite{Joh} for
computational examples. The formula for the degree of the edge surface $\mathcal E(X)$ 
appears in \cite[Theorem 2.1]{RS2}. The proof given there is based on {\em De
Jonqui\'eres' Formula}. The degree of the trisecant surface $\mathcal{D}(X)$
is due to Berzolari who first found it in 1895. One finds Berzolari's
formula in Bertin's article \cite{Ber} on the geometry of 
$\mathcal{D}(X)$.
\end{proof}

\begin{table}[h]
  \centering
  \begin{tabular}{cccccc}
    \Xhline{2\arrayrulewidth}
$d$ & $g$ && $ {\rm deg}(\mathcal{T}(X)) $& 
$ {\rm deg}(\mathcal{E}(X))$ & 
$ {\rm deg}(\mathcal{D}(X))$ \\
\hline
3 & 0 && 4 & 0 & 0 \\
4 & 0 && 6 & 6 & 2 \\
4 & 1 && 8 & 8 & 0 \\
5 & 0 && 8 & 16 & 8 \\
5 & 1 && 10 & 20 & 5 \\
5 & 2 && 12 & 24 & 2 \\
6 & 0 && 10 & 30 & 20 \\
6 & 1 && 12 & 36 & 16 \\
6 & 2 && 14 & 42 & 12 \\
6 & 3 && 16 & 48 & 8 \\
6 & 4 && 18 & 54 & 4 \\
\Xhline{2\arrayrulewidth}
\end{tabular}
\caption{\label{tab:curves}  
Degrees of the components of the
visual event surface of a space curve}
\end{table}

Table \ref{tab:curves} summarizes the conclusion of Theorem~\ref{thm:curves}
for space curves of degree $d$ at most six. The genus $g$ ranges from $0$ to
{\em Castelnuovo's bound}. Note that, for fixed $d$ and increasing $g$, the
degree of $\mathcal{D}(X)$ decreases  while that of the others increases. In
particular, there is no trisecant surface for twisted cubic curves and
elliptic quartic curves (cf.~\cite[Proposition 1]{Ber}).

The edge surface $\mathcal{E}(X)$ is of importance in convex geometry
because the algebraic boundary of the convex hull of 
a real affine curve $X_\RR$ consists of 
$\mathcal{E}(X)$ and the tritangent planes of $X$.
This was shown by Ranestad and Sturmfels  in
\cite[\S 2]{RS2}, and in \cite[\S 3]{RS2} they describe a method for  computing
$\mathcal{E}(X)$ when $X$ is rational. This theme was picked up
by Seigal and Sturmfels  in their study of real tensor decompositions
\cite{SS}. According to \cite[\S 3]{SS}, the real rank two boundary of $X$
is the union $\mathcal{T}(X) \cup \mathcal{E}(X)$, so it is part
of the visual event surface of $X$.
The study of curves in the present section is thus a
further development of \cite{RS2, SS}.

Consider now the projection $\pi_z: X \subset \PP^3\dashrightarrow \PP^2$ from
a  center $z \in \PP^3 \backslash X$. The target $\PP^2$ has two intrinsic
realizations. These live in the ambient spaces ${\rm Gr}(1,\PP^3)$ and
$(\PP^3)^*$ respectively. The first is the surface $\alpha(z)$ of all lines
in  $\PP^3$ that contain $z$. The second is the plane $z^\vee$ of all planes
in $\PP^3$ that contain $z$. Basic projective geometry yields the following
characterizations of the image curve $C_z(X)$ in these intrinsic realizations of the image plane.

\begin{proposition} \label{prop:visualcone} The image $C_z(X)$ of our curve
 $X$ is projectively equivalent to the curve $\alpha(z) \cap {\rm Ch}(X) $ in
 the Grassmannian $ {\rm Gr}(1,\PP^3)$. The dual curve $(C_z(X))^\vee$ in $
 (\PP^2)^*$ is projectively equivalent to the curve $z^\vee \cap X^\vee$ in
 the dual projective space $(\PP^3)^*$.
\end{proposition}

In computer vision, the term {\em visual cone} is used for the union of all
lines in the pre-image of a set in $\PP^2$. The visual cone of a curve is a
developable surface in $\PP^3$. The dual of the visual cone associated with
$C_z(X)$ is the plane curve $z^\vee \cap X^\vee$. Hence,
Proposition~\ref{prop:visualcone} describes the curves in ${\rm Gr}(1,\PP^3)$
and $(\PP^3)^*$ that encode this visual cone, as discussed after Theorem
\ref{thm:everydev}.

Our task in this section is to solve the following computer algebra problem. Given a general space
curve $X$, compute the ruled surfaces $\mathcal{T}(X)$, $\mathcal{E}(X)$ and
$\mathcal{D}(X)$. Here the input is either the ideal of $X$, or a parametric
representation of $X$. The output is the defining polynomial $f$ of the
surface $\mathcal{S}_C$ in $\PP^3$. If the polynomial $f$ is too
large, we compute the ideal of the curve $C$ in ${\rm Gr}(1,\PP^3)$ or,
when $\mathcal S_C$ is developable, the ideal of the dual
$\mathcal S_C^\vee$ in~$(\PP^3)^*$.

\smallskip

We use the symbols $\mathcal{T}^\ell(X)$,
$\mathcal{E}^\ell(X) $ and $\mathcal{D}^\ell(X)$
to denote the curves in the Grassmannian ${\rm Gr}(1,\PP^3)$
that represent the surfaces $\mathcal{T}(X)$, $\mathcal{E}(X)$
and $\mathcal{D}(X)$.  The degrees of these curves in $\PP^5$
are the numbers in Theorem \ref{thm:curves}.
Two of the curves appear on the lower right in Figure~\ref{fig:curve-graph}.

The surfaces $\mathcal{T}(X)$ and $\mathcal{E}(X)$ are developable, so 
they can also be represented by their dual curves in $(\PP^3)^*$.
We use the same notation as the left diagram in Figure~\ref{fig:curve-graph}:
$$
\mathcal{T}^p(X) \,:= \,\mathcal{T}(X)^\vee \qquad {\rm and} \qquad
\mathcal{E}^p(X)\, :=\, \mathcal{E}(X)^\vee.
$$
Here  the index ``$p$'' stands for planes. The earlier
used upper index ``$\ell$'' stands for lines.
The trisecant surface $\mathcal{D}(X)$ is ruled but not developable,
so it has no associated curve in $(\PP^3)^*$.

Many space curves arising in applied contexts lie in the zero set of a quadratic polynomial.
A generic curve has this property when its genus $g$ is maximal with respect
to the Castelnuovo bound~\cite[Thm. 6.4, page 351]{Har}. We now focus on that
special case. Later in this section we address our computational task for
general curves that are not on a quadric.

Let $X$ be a general curve on a smooth quadric $Q$ in $\PP^3$. 
Any line $L$ in $\PP^3$ that intersects the quadric $Q$ in three points must
lie on $Q$. Hence $L$ lies in one of the two rulings of $Q$.

\begin{remark} \label{rem:ruling}
If $d \geq 4$ and $(d,g) \not= (4,1)$ and
 $X$ lies on a quadric $Q$, then the trisecant surface $\mathcal{D}(X)$
coincides with the quadric surface $Q$, taken with an appropriate multiplicity.
\end{remark}

To derive that multiplicity, and to set the stage for computing
$\mathcal{T}(X)$ and $\mathcal{E}(X)$,  we perform a linear change of coordinates
in $\PP^3$ so that the equation of $Q$ equals $x_1 x_4 = x_2 x_3$.
Thus we identify $Q$ with the Segre surface $\PP^1 \times \PP^1$.
We fix affine coordinates $((1:s), (1:t))$.

\begin{corollary}
\label{cor:curves}
If $X$ has bidegree $(a,b)$ on $Q = \PP^1 \times \PP^1 \subset \PP^3$,
then the degrees of the tangential surface $\mathcal{T}(X)$,
 the edge surface $\mathcal{E}(X)$
and the (non-reduced) trisecant surface $\mathcal{D}(X)$~are 
$$
\begin{matrix}
{\rm deg}\bigl(\mathcal{T}(X) \bigr) & = & 2ab, \\
{\rm deg}\bigl(\mathcal{E}(X) \bigr)& =  & 2ab(a+b-3),  \\
 {\rm deg}\bigl(\mathcal{D}(X) \bigr) & = & 2  \bigl(\binom{a}{3} + \binom{b}{3} \bigr). \\
 \end{matrix}
 $$
\end{corollary}

\begin{proof}
The affine polynomial $f(s,t)$ that defines $X$
has degree $a$ in $s$ and degree $b$ in $t$. Our curve $X$ has
degree $d = a+b$ and genus $g = (a-1)(b-1)$. Indeed, it is a basic fact from
toric geometry that the genus $g$
is the number of interior lattice points of the 
Newton polygon, which is a rectangle of size $a \times b$.  
Moreover, a general curve $X$ in $\PP^3$ of that degree and genus
lies on a quadric, so we can apply Theorem \ref{thm:curves}.
We substitute $d = a+b$ and $g = (a-1)(b-1)$ into the formulas
given there. This yields the formulas in Corollary \ref{cor:curves}.
\end{proof}

We now see that the
``appropriate multiplicity'' in Remark \ref{rem:ruling}
is $\binom{a}{3} + \binom{b}{3}$.
 The two summands correspond to the two rulings of $Q$.
Each line in the first ruling meets $X$ in $a$ points, so it
counts as a trisecant with multiplicity $\binom{a}{3}$, and
ditto with $b$ for the second ruling.

We shall present algorithms for computing $\mathcal{T}(X)$ and $\mathcal{E}(X)$
from the affine polynomial $f(s,t)$ that defines $X$ in $\PP^1 \times \PP^1$.
Recall that a change of coordinates is
 required in order to apply our method in situations
when $X$ is given by its ideal in $\RR[x_1,x_2,x_3,x_4]$.
We illustrate this point, and later our algorithms,
for the case when $a=3, b=2$ and hence $d=5, g=2$.

\begin{example} \label{ex:bicanonical} \rm Let $X$ be
the bicanonical embedding of a genus $2$
curve. It has degree $5$ in $\PP^3$.
The curve is arithmetically Cohen-Macaulay. Its ideal is given by
the $2 \times 2$-minors of
\begin{equation}
\label{eq:twobythree}
 \begin{pmatrix}
\ell_{11} & \ell_{12} & q_{1} \\
\ell_{21} & \ell_{22} & q_{2} 
\end{pmatrix}.
\end{equation}
The $\ell_{ij}$ are linear forms and the $q_i$
are quadratic forms, found by computing syzygies of $X$.
Assuming $\ell_{11}, \ell_{12}, \ell_{21}, \ell_{22}$ to be linearly independent,
we change coordinates and write this~as
$$ \begin{pmatrix}
x_1 & x_3  & q_{1} (x_1,x_2,x_3,x_4)  \\
x_2 & x_4 & q_{2} (x_1,x_3)
\end{pmatrix} \quad = \quad
\begin{pmatrix}
1 & t &  g_1(s,t) \\
s & st & g_2(t)
\end{pmatrix} .
$$
Here, the polynomial $g_1$ has bidegree $(2,2)$ in $(s,t)$,
the polynomial $g_2$ has degree $2$ in $t$, and we used column operations
to eliminate $x_2$ and $x_4$ from $q_2$.
One of the $2 \times 2$-minors of this $2 \times 3$-matrix is the equation
of bidegree $(3,2)$ that defines our curve in the affine plane:
$$ f(s,t)\,\, = \,\,g_2(t) - s g_1(s,t). $$
Conversely, every polynomial of bidegree $(3,2)$
in $(s,t)$ has a matrix representation (\ref{eq:twobythree}).
\hfill $\diamondsuit$
\end{example}

To compute the tangential surface $\mathcal{T}(X)$ from $f(s,t)$, we form the $3 \times 4$-matrix
\begin{equation*}
\qquad M \,\, = \,\, \begin{pmatrix}
\,1 & \,s & \,t\, & st \\
\, 0 & \! -f_t & \,f_s \,& s f_s - t f_t \\
\,x_1 & x_2 & \,x_3 \,& x_4
\end{pmatrix} , \qquad
\hbox{where}\,\,
 f_s = \frac{\partial f}{\partial s}\,\, {\rm and} \,\,
 f_t = \frac{\partial f}{\partial t}.
\end{equation*}
The first two rows of $M$ are linearly independent,
and they span  the tangent line at the 
point of $X$ corresponding to $(s,t)$.
The second row is the image of the tangent direction $(-f_t,f_s)$ of the affine
curve $\{f=0\}$ under the linear map given by the Jacobian of $\,\mathbb{C}^2 \rightarrow \mathbb{C}^3$,
$\,(s,t) \mapsto (s,t,st)$.
 Another point $(x_1:x_2:x_3:x_4)$ lies on
that tangent line in $\PP^3$ precisely when ${\rm rank}(M) = 2$.
The following ideal is generated by five polynomials in $\RR[s,t,\,x_1,x_2,x_3,x_4]$:
\begin{equation}
\label{eq:sixvar}
 \langle \,f \,\rangle +  \langle \hbox{ $3 \times 3$-minors of $M$ } \rangle .
\end{equation}
Our argument implies the following method for finding the tangential surface of degree $2ab$:

\begin{proposition} Eliminating the unknowns $s$ and $t$ from the ideal (\ref{eq:sixvar})
yields a principal ideal in $\RR[x_1,x_2,x_3,x_4]$. Its generator is the
polynomial defining the tangential surface $\mathcal{T}(X)$.
\end{proposition} 

\begin{example} \label{ex:bicanonical2} \rm
Let $d=5, g=2$ as in Example \ref{ex:bicanonical}, and fix
the curve $X \subset \PP^1  \times \PP^1$ defined~by
\begin{equation}
\label{eq:threetwocurve}
  f \,\,=\,\, s^3 t^2 + s^3 + t^2 + s + t + 1
  \end{equation}
The polynomial that defines the tangential surface $\mathcal{T}(X)$ has degree $12$. It looks like this:
$$ 
93 x_1^{12}+310 x_1^{11} x_2+341 x_1^{10} x_2^2+558 x_1^9 x_2^3+1054 x_1^8 x_2^4+744 x_1^7 x_2^5
+837 x_1^6x_2^6
+ \,\cdots\, +216 x_2 x_4^{11}+108x_4^{12}.
$$
This polynomial has $432$ monomials, out of
the $\binom{15}{3} = 455$ possible ones of degree $12$.
\hfill $\diamondsuit$
\end{example}

We now consider the edge surface $\mathcal{E}(X)$. Its
degree is $2ab(a+b-3)$. We  shall compute its dual representation $\mathcal{E}^p(X)$.
Planes in $\PP^3$ correspond to curves of bidegree $(1,1)$ in $\PP^1 \times \PP^1$:
\begin{equation}
\label{eq:oneonecurve}
 y_1+ y_2 s + y_3 t + y_4 st \,\,= \,\, 0.
 \end{equation}
We solve for $s$, substitute into $f(s,t)$, and
clear denominators. The result is a polynomial in one variable $t$
of degree $a+b$. We seek the condition that this 
 has two double roots, corresponding to $m=2+2$ in
Figure \ref{fig:curve-graph}. This condition defines the curve
$\mathcal{E}^p(X)$ in   $(\PP^3)^*$.

\begin{example} \label{ex:binonical3} \rm
Let $d=5, g=2$ as before in Example \ref{ex:bicanonical}.
Fix the curve $X$ in Example \ref{ex:bicanonical2}.
We shall compute the dual curve $\mathcal{E}^p(X)$
to the edge surface.
The moving curve  (\ref{eq:oneonecurve}) has five intersection points 
with the fixed curve $f=0$ in (\ref{eq:threetwocurve}). Their $t$-coordinates
are the roots of
\begin{equation}
\label{eq:quintic} c_0 + c_1 t + c_2 t^2 + c_3 t^3 + c_4 t^4 + c_5 t^5 \,=\, 0 , 
\end{equation}
where the $c_i$ are quadratic polynomials in $y_1,y_2,y_3,y_4$.
Regarding the coefficients as unknowns, we pre-compute the ideal
$\Delta_{(2,2)}(5) \subset \RR[c_0,c_1,c_2,c_3,c_4,c_5]$ whose variety consists
of quintics (\ref{eq:quintic}) with two double roots.
The ideal $\Delta_{(2,2)}(5)$ has codimension $2$ and degree $12$. It is generated
by $10$ quintics in the $c_i$, as seen in the row labeled 
$\lambda = 221$ in  \cite[Table 1]{LS}. Let $I$ be the ideal obtained from $\Delta_{(2,2)}(5)$
by replacing the $c_i$ with the quadrics in $y_1,y_2,y_3,y_4$
that represent the specific curve $X$, and then
saturating by the irrelevant ideal $\langle y_1,y_2,y_3,y_4 \rangle$.
The variety of $I$ is the curve $\mathcal{E}^p(X)$ in $(\PP^3)^*$.
The ideal $I$ has $14$ minimal generators, all of degree $10$,
with large integer coefficients.
This is the dual representation of the edge surface.

Computing $\mathcal{E}(X)$ by directly dualizing $\mathcal{E}^p(X)$ rarely terminates in practice.
It is easier to intersect
$\mathcal{E}(X) $ with lines or planes, as explained in Section~\ref{sec2},
around  (\ref{eq:linepara}) and (\ref{eq:binfor}).
\hfill $\diamondsuit$
\end{example}

We now consider curves $X$ that need not lie on a quadric $Q$.
Let us first assume that $X$ is the image 
of a variety $Y$ in a higher-dimensional space $\PP^d$ under a linear projection
$\alpha: \PP^d \dashrightarrow \PP^3$. This allows us to preprocess $Y$,
especially if the {\em Chow form} of $Y$ is known.

This approach works well when $X$ is rational. Here $Y$ is
the {\em rational normal curve}  in $\PP^d$, which is parameterized by
$ \bigl(\,1: t:t^2: \cdots : t^d \,\bigr)$.
Let $\alpha : \PP^d \dashrightarrow \PP^3$ be the linear projection that maps
$Y$ onto our curve $X$. We write
 $A$ for the $4 \times (d+1)$ matrix that represents $\alpha$.

 We first explain the computation of the tangential surface $\mathcal{T}(X)$. 
  Let $s$ be an unknown and let $Q$ be the skew-symmetric $4 {\times} 4$-matrix obtained from (\ref{eq:skewsymP}) by substituting to dual Pl\"ucker coordinates.
 We form the ideal in $\RR[s, q_{12},q_{13},\ldots,q_{34}] $
that is generated by the Pl\"ucker quadric  $q_{12} q_{34} - q_{13} q_{24} + q_{14} q_{23} $
and  the eight entries of the $4 \times 2$ matrix 
\begin{equation*}
Q \cdot A \cdot
\begin{pmatrix}
1 & s & s^2 & s^3 & s^4 & \cdots & s^d \\
0 & 1  & 2 s & 3s^2 & 4s^3 & \cdots & d s^{d-1}
\end{pmatrix}^{\!\! T}.
\end{equation*}
Eliminating $s$ and saturating with respect to the $q_{ij}$
 now yields the ideal of the  curve $\mathcal{T}^\ell(X)$.
From this we can compute the defining polynomial of
$\mathcal{T}(X)$ via (\ref{eq:skewsymP}).
The edge surface $\mathcal{E}(X)$ can be computed similarly.
This was also discussed in \cite[\S 2]{RS2} and  in \cite[\S 3]{SS}.

We now compute
the trisecant surface $\mathcal{D}(X)$ of a rational curve $X$ of degree $d$ in $\PP^3$.
The Chow form of $Y$ is the resultant of two binary forms of degree $d$.
We write this as the determinant of the {\em B\'ezout matrix} $B(r)$.
This is a symmetric $d \times d$-matrix whose entries are linear forms in the Pl\"ucker
coordinates $r_{ij}$ of $(d-2)$-planes in $\PP^d$.
For the formula we refer to equation (1.18) on page 402 in Section III.12.1  of \cite{GKZ}.
The  B\'ezout matrix for $d=6$~equals
\begin{equation*}
 B(r) \,\,\, = \,\,\,
\begin{pmatrix}
r_{12}    & r_{13}  & r_{14} &         r_{15} &        r_{16} &    r_{17} \\
r_{13} & r_{14}+r_{23} &  r_{15}+r_{24} &     r_{16}+r_{25} &     r_{17}+r_{26} & r_{27} \\
r_{14} & r_{15}+r_{24} &  r_{16}+r_{25}+r_{34} & r_{17}+r_{26}+r_{35}  & r_{27}+r_{36} & r_{37} \\
r_{15} & r_{16}+r_{25} & r_{17}+r_{26}+r_{35} & r_{27}+r_{36}+r_{45} &  r_{37}+r_{46} & r_{47} \\
r_{16} &  r_{17}+r_{26} & r_{27}+r_{36} &    r_{37}+r_{46} &    r_{47}+r_{56} & r_{57}  \\
r_{17} & r_{27} &     r_{37} &         r_{47} &         r_{57} &    r_{67}
\end{pmatrix} \quad
\end{equation*}

We shall use the following fact that is well-known in computer algebra; see
\cite[page 1228]{ADG}.

\begin{lemma} \label{lem:lalo}
The minors of the B\'ezout matrix $B(r)$ having size $d-k+1$ define
an irreducible variety of codimension $k$ in the Grassmannian
of $(d-2)$-planes in $\PP^d$. General points $q$ on this 
variety represent pairs of univariate polynomials
of degree $d$ that have $k$ common zeros.
\end{lemma}

The second exterior power $\wedge_2 A$
of the matrix $A$ is a matrix of format  $6 \times \binom{d+1}{2}$.
We write $p = (p_{12}, p_{13},p_{14},p_{23},p_{24},p_{34})$ for the
Pl\"ucker coordinates of a line in $\PP^3$.  The preimage
of the line $p$ under the projection $\alpha$ is the $(d-2)$-plane in $\PP^d$
with Pl\"ucker coordinates 
$r \,=\, p \cdot \wedge_2 A $.

\begin{proposition}
If the $4 \times (d{+}1)$-matrix A is sufficiently general,  then the ideal of 
$(d{-}2) \times (d{-}2)$-minors of the matrix
$ B \bigl( p \cdot \wedge_2 A \bigr) $ defines the
curve $\mathcal{D}^\ell(X)$ of degree  $2 \binom{d-1}{3}$ in $\PP^5$.
\end{proposition}

\begin{proof} We use
Lemma~\ref{lem:lalo}. Intersecting the curve $X$ with two planes
in $\PP^3$ amounts to solving two univariate
polynomials of degree $d$.  The Chow form of $X$ is
the determinant of the specialized B\'ezout matrix 
$ B \bigl( p \cdot \wedge_2 A \bigr) $. The corank of that matrix 
is the number of common zeros. That number is $ 3$ when
the intersection of the two planes is a trisecant line of $X$. Thus, the line $p$
 is in the trisecant curve $\mathcal{D}^\ell(X)$
 precisely when $ \,B \bigl( p \cdot \wedge_2 A \bigr)\, $ has rank $\leq d-3$.
By Lemma \ref{lem:everyruled} and setting
$g=0$ in Theorem \ref{thm:curves},  the degree of
$\mathcal{D}^\ell(X)$ is $2 \binom{d-1}{3}$.
\end{proof}

For an illustration consider rational curves $X$ of degree $d=6$.  Then $A$ is a 
$4 \times 7$ matrix, and $\wedge_2 A$ is a $ 6 \times 21$ matrix.
The ideal of $4 \times 4$-minors of the $6 \times 6$-matrix
$ \,B \bigl( p \cdot \wedge_2 A \bigr)\, $ is minimally generated,
modulo the Pl\"ucker relation, by $45$ quartics.
After saturating with respect to the irrelevant
maximal ideal $\langle p_{12}, p_{13}, p_{14},
p_{23}, p_{24},p_{34} \rangle$, we obtain
the prime ideal $I_C$ of the curve $C$ of trisecants.  The ideal $I_C$ has
degree $20$ and is generated by $10 $ cubics.

The success of this computation relied on writing
the Chow form of $X$ as a determinant of a matrix whose entries are linear in Pl\"ucker coordinates.
Eisenbud, Schreyer and Weyman \cite[\S 4]{ESW} proved that such a formula exists for
all curves.
 See \cite[Proposition 4.2]{ESW} for a derivation of the
B\'ezout matrix $B(r)$ from the perspective of Ulrich sheaves, and 
\cite[Example 4.6]{ESW} for an extension to hyperelliptic curves.
Whenever we have such matrices explicitly, we get
the surface $\mathcal{D}(X)$ by imposing 
the corank $3$ constraints. Such
 matrix formulas for Chow forms can also
be derived for curves $X$ in $\PP^3$ that arise by intersecting certain nice varieties.

\begin{example} \rm
Let $X$ be the curve in $\PP^3$ defined by the
$3 \times 3$-minors of the $3 \times 4$-matrix 
$$ M(x) \quad = \quad \begin{pmatrix}
\,x_1+x_4 & x_2 -x_1 & x_3-x_2 & x_3+x_4 \\
\,x_4-x_3& x_1+x_4 & x_2-x_1 & x_2+x_3 \\
\,x_3-x_2 & x_2-x_3+x_4 & x_1+x_4 & x_1 
\end{pmatrix}.
$$
This curve has $(d,g) = (6,3)$.
By computing syzygies, we can represent every curve
of degree $6$ and genus $3$ via such a matrix with linear entries.
This follows from the {\em Hilbert-Burch Theorem}.
Let $Y$ be the variety of
$3 \times 4$-matrices of rank $\leq 2$.
The Chow form of $Y$ is the determinant of the following $6 \times 6$-matrix
in dual Pl\"ucker coordinates for lines in $\PP^{11}$:
\begin{equation}
\label{eq:bondal}
\begin{small}
 \begin{pmatrix}
r_{11,12} &  r_{11,22}+r_{21,12} &  r_{11,32}+r_{31,12} & r_{21,22} & r_{21,32}+r_{31,22} & r_{31,32} \\
r_{11,13} &  r_{11,23}+r_{21,13} &  r_{11,33}+r_{31,13} & r_{21,23} & r_{21,33}+r_{31,23} & r_{31,33} \\
r_{11,14} &  r_{11,24}+r_{21,14} &  r_{11,34}+r_{31,14} & r_{21,24} & r_{21,34}+r_{31,24} & r_{31,34} \\
r_{12,13} & r_{12,23}+r_{22,13} &  r_{12,33}+r_{32,13} & r_{22,23} & r_{22,33}+r_{32,23} & r_{32,33} \\
r_{12,14} & r_{12,24}+r_{22,14} &  r_{12,34}+r_{32,14} & r_{22,24} & r_{22,34}+r_{32,24} & r_{32,34} \\
r_{13,14} & r_{13,24}+r_{23,14} & r_{13,34}+r_{33,14}  & r_{23,24} & r_{23,34}+r_{33,24} & r_{33,34}
\end{pmatrix}.
\end{small}
\end{equation}
This matrix appears in \cite[page 472]{GKZ}. We now replace the Pl\"ucker coordinates 
$r_{ij,kl}$ by linear forms in the six coordinates 
$q_{ij} = a_i b_j - b_i a_j$ for lines in $\PP^3$. 
For instance,
$r_{11,12 } = q_{12}+q_{14}-q_{24}$,
$ r_{11,13} = -q_{12}+q_{13}+q_{24}-q_{34}$,
$ r_{11,14} = q_{13}+q_{14}-q_{34}$, $\ldots$.
These linear forms are obtained by setting
 $r_{ij,kl} = M(a)_{ij} M(b)_{kl} - M(b)_{ij} M(a)_{kl}$
where $a= (a_1,a_2,a_3,a_4)$ and $b=(b_1,b_2,b_3,b_4)$.

The trisecant curve $\mathcal{D}^\ell(X)$ is defined in 
${\rm Gr}(1,\PP^3) \subset \PP^5 $ by the $4 \times 4$-minors of the resulting matrix (\ref{eq:bondal}).
Saturating by $\langle q_{12}, q_{13}, q_{14},q_{23},q_{24},q_{34} \rangle $ yields
the prime ideal of $\mathcal{D}^\ell(X)$:
$$ \begin{small} \begin{matrix} I \,= \, \langle
q_{14} q_{23}-q_{13} q_{24}+q_{12} q_{34},\,
   q_{13} q_{23}-q_{23}^2-q_{14} q_{24}-q_{24}^2+q_{12} q_{34}-2 q_{13} q_{34}+
   2 q_{14} q_{34}+2 q_{23} q_{34},\, \\ \qquad
      q_{12} q_{23}-q_{23}^2-q_{12} q_{24}+q_{13} q_{24}-3 q_{14} q_{24}-q_{23} q_{24}
      -q_{12} q_{34}+2 q_{14} q_{34}{+}3 q_{23}  q_{34}{+}q_{24} q_{34}{-}2 q_{34}^2, \\ \quad
      q_{14}^2{+}q_{14} q_{24}{+}q_{24}^2{-}q_{12} q_{34}{-}2 q_{14} q_{34}{-}q_{23} q_{34}{+}q_{34}^2,\,
      q_{13}^2{-}q_{23}^2{-}q_{14} q_{24}{-}2 q_{13} q_{34}{+}q_{14} q_{34}{+}2 q_{23} q_{34}, \\
 q_{12} q_{14}-2 q_{14} q_{24}-q_{23} q_{24}-2 q_{24}^2+q_{12} q_{34}+2 q_{14} q_{34}+
 2 q_{23} q_{34}+2 q_{24} q_{34} -2 q_{34}^2, \\
   q_{12} q_{13}+2 q_{13} q_{14}-q_{23}^2-q_{12} q_{24}-q_{13} q_{24}-2 q_{14} q_{24}
  -q_{23} q_{24}+3 q_{23} q_{34}+2 q_{24}      q_{34}-2 q_{34}^2 \rangle .
      \end{matrix}
      \end{small}      $$
      From this, we easily find
      the octic equation of the trisecant surface $\mathcal{D}(X)$:
      $$       x_1^7 x_3-2 x_1^4 x_2^3 x_3+x_1 x_2^6 x_3
      +2 x_1^5 x_2 x_3^2 +2 x_1^4 x_2^2 x_3^2
      +2 x_1^3 x_2^3 x_3^2-2 x_1^2 x_2^4 x_3^2 + \cdots
      $$
   This polynomial uses $136$ of the $165 = \binom{8+3}{3} $  monomials of degree eight.
\hfill $\diamondsuit$
 \end{example}

The past few pages were devoted to specialized techniques that exploit the 
structure of a given curve $X$. Such techniques can be designed for
all entries in Table~\ref{tab:curves}. However, equally important are
general purpose methods that work for all curves. We close this
section by discussing the latter. The curve $X$ is given by its
ideal $I = \langle f_1,f_2,\ldots, f_k\rangle$ in $\RR[x_1,x_2,x_3,x_4]$.

The  edge surface $\mathcal{E}(X)$ was already discussed in
\cite{RS2, SS}. We therefore focus on the other two ruled surfaces
in Theorem~\ref{thm:curves}. The easier among them is the tangential surface $\mathcal{T}(X)$.
At any given point $p $ on the curve $X$, the tangent
line is defined by the linear equations
\begin{equation}\label{eq:tangential}
\begin{pmatrix} \frac{\partial f_1}{\partial x_1}(p) & \frac{\partial
f_1}{\partial x_2}(p) & \frac{\partial f_1}{\partial x_3}(p) & \frac{\partial
f_1}{\partial x_4}(p)\\
 \vdots & \vdots & \vdots & \vdots \\
\frac{\partial f_k}{\partial x_1}(p) & \frac{\partial f_k}{\partial x_2}(p) &
\frac{\partial f_k}{\partial x_3}(p) & \frac{\partial f_k}{\partial x_4}(p)
\end{pmatrix}  \cdot
\begin{pmatrix} x_1\\ x_2 \\ x_3 \\ x_4
\end{pmatrix} \quad = \quad \begin{pmatrix} 0 \\ 0 \\ 0 \\ 0 \end{pmatrix}.
\end{equation} 
To find the polynomial $F$ defining $\mathcal{T}(X)$, we take a
vector of variables $p = (y_1,y_2,y_3,y_4)$, and we augment $I$ with the constraints
\eqref{eq:tangential}. This gives an ideal in
 $\RR[x_1,x_2,x_3,x_4,y_1,y_2,y_3,y_4]$. From that ideal, we
saturate and eliminate the variables $y_1,y_2,y_3,y_4$.
The output is $\langle F \rangle$.

The trisecant surface $\mathcal{D}(X)$ will be represented by its curve
$\mathcal{D}^\ell(X) \subset {\rm Gr}(1,\PP^3)$. To compute this,
we parameterize the line in $\PP^3$ as in (\ref{eq:lineparaone}).
Suppose  for now that our curve is a complete intersection:
$X = \{f_1=f_2 =0\}$.
We want the univariate polynomials
$f_1(z(t))$ and $f_2(z(t))$ to have three common roots, i.e.,~their
greatest common divisor (GCD) has degree $\geq 3$. This 
 can be expressed using {\em subresultants} \cite{ADG}.  The vanishing of all
 subresultants of order $i=0,\ldots, r-1$ for two polynomials in $t$  means that their GCD has degree at
least~$r$.  In our case, we form the ideal given by the subresultant
coefficients of $f_1(z(t))$ and $f_2(z(t))$ of order $0,1$ and $2$ (together with
the Pl\"ucker relation). The ideal of the trisecant curve $\mathcal D^\ell(X)$
is obtained by saturating by the ideal of the leading coefficients of $f_1(z(t))$ and $f_2(z(t))$.

This approach generalizes to the case when $X$ is not a complete
intersection. Indeed, if $X$ is defined by $f_1,\ldots,f_k$, then we can use
the same strategy to impose that $s_1 f_1(z(t)) + \cdots + s_{k-1}
f_{k-1}(z(t))$ and $f_k(z(t))$ have three roots in common for any choice of
$s_1,\ldots,s_{k-1}$.

We conclude this section with an example that illustrates the last row of Table \ref{tab:curves}.

\begin{example} \rm
Let $X$ be the smooth curve of degree $6$ and genus $4$ in $\PP^3$ defined by
$$   x_1^2+x_2^2+x_3^2+x_4^2 \,\, = \,\,
 x_1^3+x_2^3+x_3^3+x_4^3 \,\, = \,\,  0.
$$
The above method easily yields the equation of degree $18$ for the tangential surface
$\mathcal{T}(X)$:
$$ \begin{small} \!
4 x_1^{12}x_2^6-12 x_1^{12} x_2^5 x_3-12 x_1^{12} x_2^5 x_4+21 x_1^{12} x_2^4 x_3^2  + 
\cdots + \underbar{13770} x_1^6 x_2^4 x_3^4 x_4^4
+ \cdots
+24 x_3^7 x_4^{11} + 4 x_3^6 x_4^{12}.
\end{small}
$$
This polynomial has $1094$ terms. Its largest coefficient is underlined.
A compact encoding is given by the $11$ quadratic generators of the
 ideal of $\mathcal{T}^\ell(X)$. It is also easy to compute the
quartic surface $\mathcal{D}(X)$, and it takes a little longer to
compute the degree $54$ curve $\mathcal{E}^\ell(X)$.
\hfill $\diamondsuit$
\end{example}

\section{Views of Surfaces}
\label{sec4}

We now turn to the visual events for a general surface $X$ in projective
$3$-space $\PP^3$. The six visual events associated with  $X$ were mentioned
in the Introduction in items (L) and (M). We shall explain these events and
how they give rise to the following five irreducible surfaces:

\begin{enumerate}
\item The {\em flecnodal surface} $\mathcal{F}(X)$ is the union of all lines $L$ with
contact of order $4$ at a point of~$X$. In other words, the equation of $X$ 
restricted to $L$ has a root of multiplicity~$4$.
\item The {\em cusp crossing surface} $\mathcal{C}(X)$ is the union of all lines $L$
with contact of order $3+2$ at two points of $X$, i.e.,~the equation for $X \cap L$ on $L$ has a 
triple root and a double~root.
\item The {\em tritangent surface} $\mathcal{T}(X)$ is the union of all lines $L$
with contact of order $2+2+2$ at three points of $X$, i.e.,~the equation for $X \cap L$ on $L$ has 
three double roots.
\item The {\em edge surface} $\mathcal{E}(X)$ is the envelope of the bitangent planes of $X$.
It is the union of all bitangent lines arising from these planes. This surface was denoted $(X^{[2]})^\vee$ in \cite{RS2}.
\item The {\em parabolic surface} $\mathcal{P}(X)$ is the envelope of all tangent planes that have contact of order $3$ with $X$. It is the union of all principal tangents at parabolic points \cite[\S A.1.2]{PPK}.
\end{enumerate}

The following theorem characterizes the structure of the visual event surface for $X$ in $ \PP^3$.

\begin{theorem}
\label{thm:surfaces}
For a general surface $X$ of degree $d$ in $\PP^3$, the visual event surface
$\mathcal{V}(X)$ decomposes into the five components listed above. The degrees 
of these surfaces are:
$$
\begin{matrix}
{\rm deg}\bigl(\mathcal{F}(X) \bigr) & = & 2d(d-3)(3d-2) ,&  \hbox{\rm \cite[\S 597]{Sal} and \cite[Prop.~4.5]{Pet}} \\
{\rm deg}\bigl(\mathcal{C}(X) \bigr)& =  & d (d-3)(d-4)(d^2+6d-4),  & \,\, \hbox{\rm \cite[\S 598]{Sal} and \cite[Prop.~4.12]{Pet}} \\
 {\rm deg}\bigl(\mathcal{T}(X) \bigr) & = &  \frac{1}{3} d(d-3)(d-4)(d-5)(d^2+3d-2), \,\, &  \,\,\hbox{\rm \cite[\S 599]{Sal}
  and \cite[Prop.~4.10]{Pet}} \\
 {\rm deg}\bigl(\mathcal{E}(X) \bigr) & = &  d(d-2)(d-3)(d^2+2d-4) , & \, \,\hbox{\rm \cite[\S 613]{Sal} and \cite[Prop.~4.16]{Pet}} \\
 {\rm deg}\bigl(\mathcal{P}(X) \bigr) & = &   2 d(d-2)(3d-4).  &  \hbox{\rm \cite[\S 608]{Sal} and \cite[Prop.~4.3]{Pet}}
 \end{matrix}
 $$
 \end{theorem}

We first learned these degree formulas from Petitjean's article \cite{Pet}.
We then discovered that all five formulas already appeared in
Salmon's 1882 book \cite{Sal}.  The precise pointers to both sources are given
above. In Section~\ref{sec5} we present new proofs that are self-contained,
except for pointers to the  textbook~\cite{EH}. We here derive the
first assertion in Theorem~\ref{thm:surfaces}, that is, we justify the five
irreducible surfaces. See the end of this section.

\smallskip

The five ruled surfaces in Theorem~\ref{thm:surfaces} are encoded by the
curves shown in the last row in Figure~\ref{fig:surface-graph}. The surfaces
$\mathcal{E}(X)$ and $\mathcal{P}(X)$ are developable, and are the duals of
the singular loci shown on the left in Figure~\ref{fig:surface-graph}. The
remaining three surfaces $\mathcal T(X)$, $\mathcal C(X)$ and $\mathcal F(X)$
arise from the curves in the Grassmannian ${\rm Gr}(1,\PP^3)$ 
seen on the right of that diagram.

The curves of lines and planes from Figure~\ref{fig:surface-graph} 
capture both the local and multi-local~features of the
surface $X$. This is an advantage compared to the traditional approach for
studying the appearance of surfaces based on differential geometry and
singularity theory. In the computer vision literature \cite{BD, PPK, PH,
Pon90, Rie92}, prominent local features of a surface were defined in terms of
the {\em euclidean Gauss map} and the {\em asymptotic spherical image}. These
are maps from the surface to the unit sphere $\mathbb{S}^2$, taking a point on
$X_\RR$ to its normal direction, or to the direction of one of its principal
tangents. In our algebro-geometric setting, the role of $\mathbb{S}^2$ is
played by the dual surface $X^\vee
\subset (\PP^3)^*$ and the principal tangent surface $PT(X)\subset {\rm
Gr}(1,\PP^3)$. These surfaces carry much more information than  the unit
sphere $\mathbb{S}^2$.

Consider now the projection $\pi_z: X \subset \PP^3\dashrightarrow \PP^2$ 
from a  center $z \in \PP^3 \backslash X$. The following result, 
analogous to Proposition~\ref{prop:visualcone}, describes intrinsic
realizations of the contour $C_z(X)$.

\begin{proposition} \label{prop:visualcone_surf}
 The contour $C_z(X)$ of our surface $X$
is projectively equivalent to the curve $\alpha(z) \cap {\rm Hur}(X) $
in the Grassmannian $ {\rm Gr}(1,\PP^3)$. 
The curve $(C_z(X))^\vee$ in $(\PP^2)^*$ that is dual to the contour
is projectively equivalent to
the curve $z^\vee \cap X^\vee$ in the dual projective space $(\PP^3)^*$.
\end{proposition}

We now describe the singularities of the  curve $C_z(X)$ that arises by projecting $X$ from $z$.
Given a point $u  \in C_z(X)$, we write  $L_u = \pi_z^{-1}(u) \in {\rm
Gr}(1,\PP^3)$ for its fiber under $\pi_z$.
For a general viewpoint $z$, the plane $\alpha(z)$ intersects the
Hurwitz threefold ${\rm Hur}(X)$ transversely:
\begin{itemize}
	\item If $L_u$ is a general point of ${\rm Hur}(X)$, then $u$ is a smooth
	point of the curve $C_z(X)$. \vspace{-0.1in}
	\item If $L_u$ is a general point of ${\rm Bit}(X) \subset {\rm Hur}(X)$,
	then $u$ is a simple node of $C_z(X)$. \vspace{-0.1in}
	\item If $L_u$ is a general point of ${\rm PT}(X) \subset {\rm Hur}(X)$,
	then $u$ is a simple cusp of $C_z(X)$.
\end{itemize} 

These singularities exist for any viewpoint $z$, since
$\alpha(z)$ always intersects the surfaces ${\rm Bit}(X)$ and
${\rm PT}(X)$ in ${\rm Gr}(1,\PP^3)$. 
It is also interesting to learn (e.g.~from \cite{PH})
that if $L_u$ is a (non-principal) tangent line at a
parabolic point, then $u$ is a flex point of $C_z(X)$.
We are interested in higher order singularities seen in the image curve 
for special viewpoints:
\begin{itemize}
	\item[($\mathcal{T}$)] If $L_u$ is a tritangent line, then $u$ is a triple point.
	 This is a {\em triple point} event.
	\item[($\mathcal{C}$)] If $L_u$ is a principal bitangent, then $C_z(X)$ 
	has a smooth branch and a cuspidal branch that meet at $u$.
	This is a {\em cusp crossing} event.
	\item[($\mathcal{F}$)] If $L_u$ is a flecnodal line, then $u$ is the limit of two cusps
	and a node, i.e.,~an infinitesimal change of the viewpoint produces
	two cusps and a node. This is a {\em swallowtail} event.
	\item[($\mathcal{E}$)] If $L_u$ is a bitangent line on a bitangent plane, then $u$ is a tacnode.
	It is obtained as the limit of two smooth branches coming together at $u$. This is a {\em tangent
	crossing} event.
    \item[($\mathcal{P}$)] If $L_u$ is the principal tangent at a parabolic
    point $p$, then, over the real numbers, two behaviors are possible: either
    $u$ is an isolated node, which corresponds to a {\em lip} event, or $u$ is a
    tacnode, obtained as the limit of two cusps, which is a {\em beak-to-beak}
    event.
\end{itemize}

The triple point, cusp crossing, and tangent crossing events are multi-local.
The six visual events are shown in Figure~\ref{fig:events}. Detailed
renderings of these pictures are ubiquitous in the relevant computer vision
literature. For instance, see \cite[Figures 5 and 6]{PPK}, and Figures 13.20
through 13.25 in the textbook~\cite{ForPon}. For a fixed general surface $X$
in $\PP^3$, the locus of exceptional viewpoints consists of the five ruled
surfaces $\mathcal{F}(X),
\mathcal{C}(X),
\mathcal{T}(X),
\mathcal{E}(X) $ and 
$\mathcal{P}(X)$.
These were  defined both in Figure \ref{fig:surface-graph} and at the beginning of Section \ref{sec4}.

\begin{figure}[htbp]
	\centering
	\includegraphics[width=0.45\textwidth]{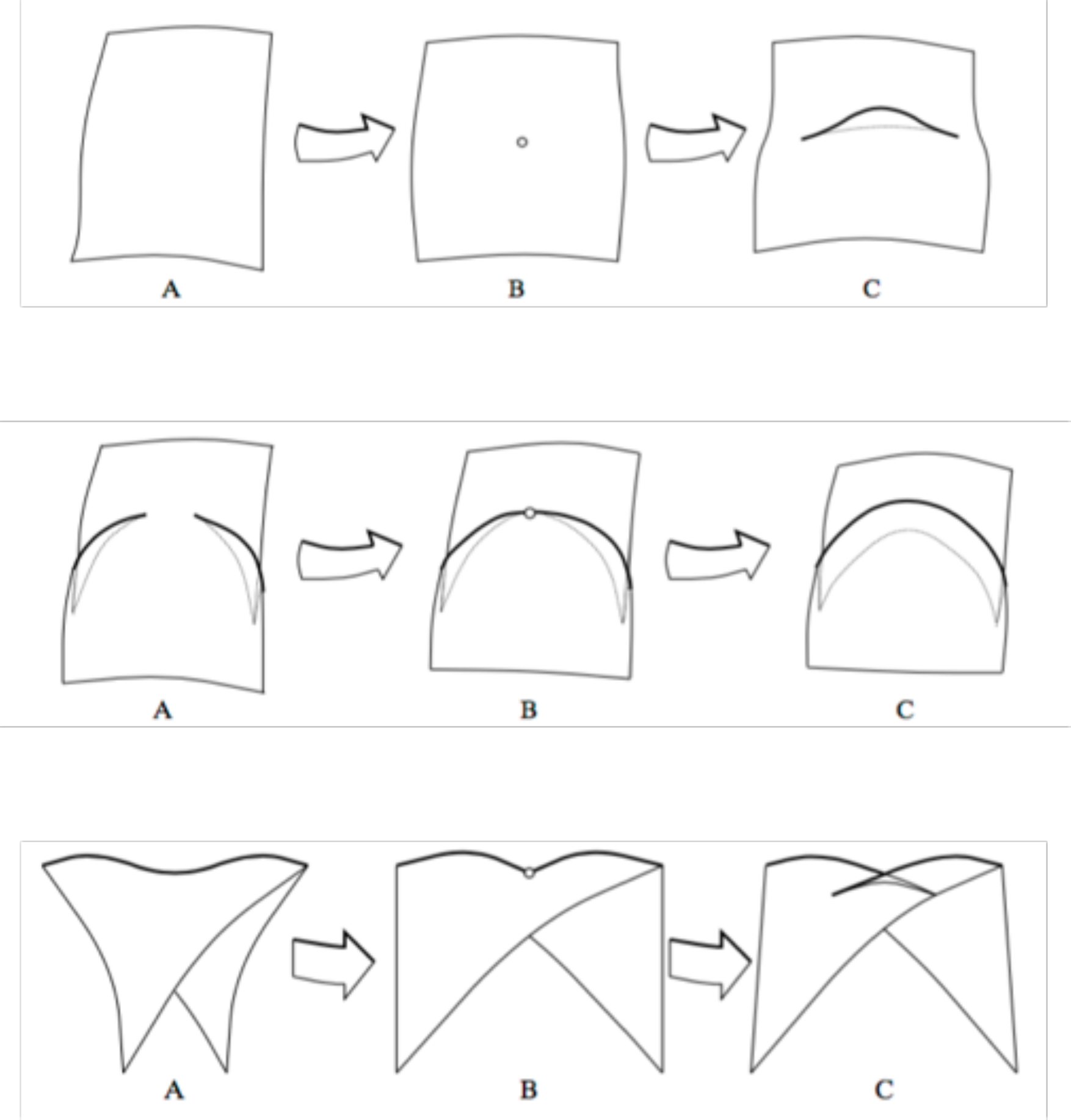} \quad
	\includegraphics[width=0.5\textwidth]{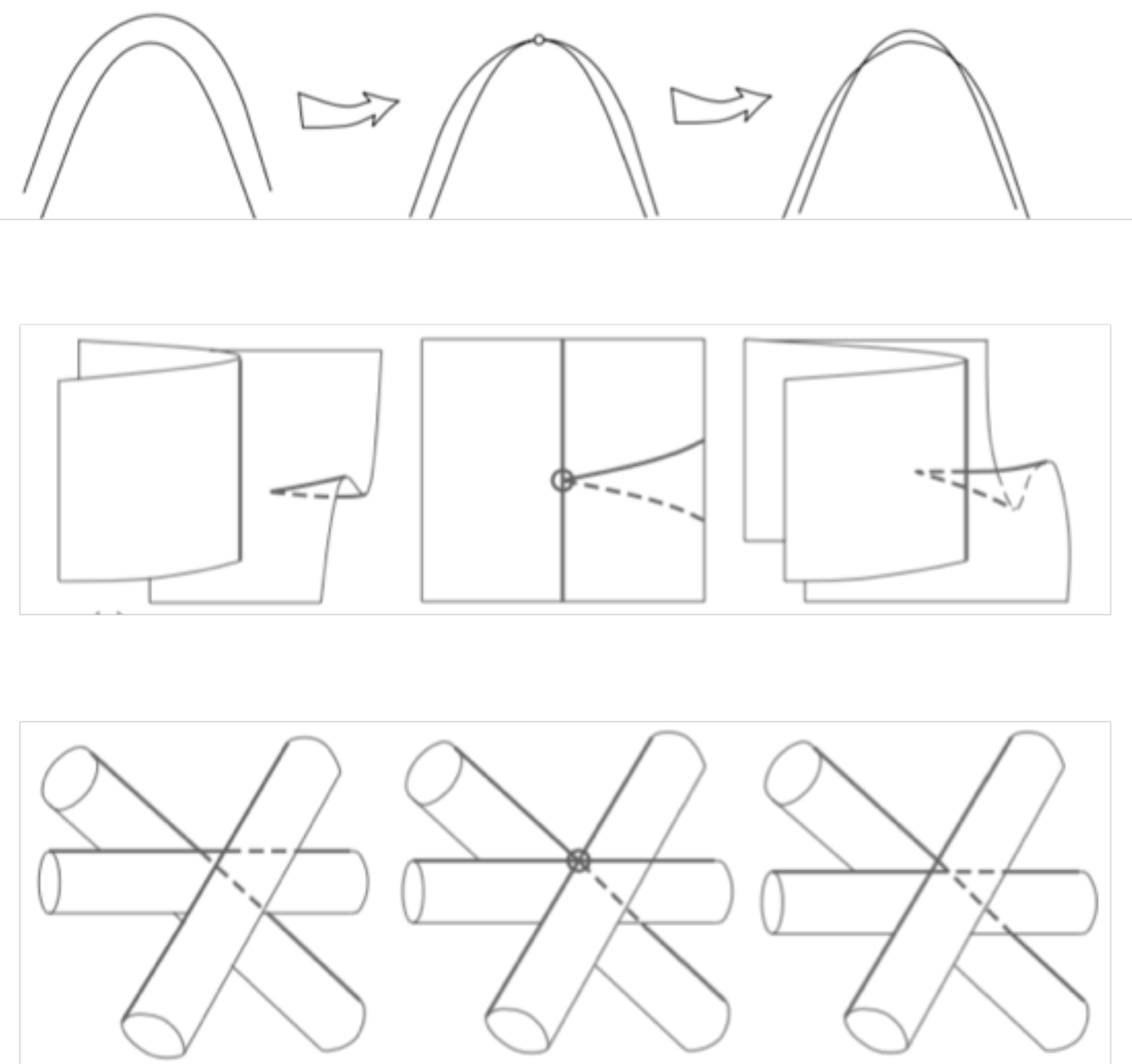}
	\caption{\small The catalogue of visual events
	for the projections of a smooth surface from a viewpoint that moves.
The local events ({\em left},
from top to bottom) are lip, beak-to-beak, swallowtail. The multi-local events 
({\em right}, from top to bottom) are tangent crossing, cusp crossing, triple
point. Reprinted from~\cite{Pae} with permission of Springer.}
	\label{fig:events}
\end{figure}

We now briefly explain how to distinguish the two possible local behaviors (lip
versus beak-to-beak) of the contour when the viewpoint $z$ belongs to the
parabolic surface $\mathcal P(X)$. As argued in Section~\ref{sec2}, the
parabolic surface $\mathcal P(X)$ is a developable surface, since it is dual
to the curve $\mathcal P^p(X)$ in $(\PP^3)^*$. In particular, all principal
tangents at parabolic points are the tangents of the edge of regression curve,
denoted by $E(\mathcal P(X))$. This allows us to associate each parabolic
point $x$ with another point $e_x$, where the principal tangent at $x$ is
tangent to $E(\mathcal P(X))$. In real projective $3$-space, the complement of
$\{x,e_x\}$ in that line has two connected components.  The distinction
between lip and beak-to-beak is made by which of these two components the
viewpoint $z$ belongs to. This was shown in \cite[Theorem 4.10]{Plat}.


\smallskip

We next offer an informal discussion that should provide an
intuitive understanding of our five event surfaces. The
following are some real life situations where these events can actually be
observed. We encourage our readers to look at the world from multiple viewpoints, 
and to then spot the six pictures of Figure~\ref{fig:events}.
Do look carefully at the objects that surround~you.

We first note that cuspidal and nodal singularities of image contours are stable
features, which are visible in most surfaces. Nodes occur whenever
occlusions create discontinuities in the contour. Cusps only appear for non-convex objects.
For instance, they can be observed on the folds of a piece of
cloth. From an exceptional viewpoint, it is possible that several of these
singularities occur along the same visual ray. This gives rise to a multi-local visual event 
(cusp crossing, tangent crossing, or triple points). 
Try it with a napkin or  towel.

The three local events on the left in Figure~\ref{fig:events} are more complicated.
Videos of these events and their corresponding ruled surfaces are available at
\begin{center}
\url{www.mis.mpg.de/nlalg/videos/changingviews}
\end{center}

It takes some practice to discover the
local events in the real world. Here are some concrete examples that we found helpful:

\begin{itemize}
     \item {\em Lip event:} If we observe a small hill from a high aerial
     viewpoint (say, from a hot air balloon), then all points on the ground
     are visible. The hill does not generate an image contour.
     However, as we descend closer to the ground, the profile of the hill 
     suddenly becomes visible in the contour. This qualitative change of appearance is 
          a lip event.

     \item {\em Beak-to-beak.} Observe a glass bottle from the bottom, with
      your eye close to the base. You see a part of the contour generated by the
      convex region where the sectional diameter of the bottle decreases. Now,
      tilt the bottle slowly towards its upright position. At some point, you
            see a complete path  from the base to the top of the bottle.
      Previously your view had been
      blocked. This is a beak-to-beak event. Contrary to the lip
      event, the contour does not disappear at the transition point, but it
      breaks into two pieces.
     \item {\em Swallowtail.} The traditional drawing of a (transparent) torus
     presents two swallowtails. We see both cuspidal and nodal singularities in the contour
      \cite[Fig.~2]{Rie92}. As we rotate the
     torus, a  visual event occurs, and these singularities
     disappear. Try it with a bagel.
 \end{itemize} 

 After having presented an intuitive descriptions of visual events, we now sketch a  formal argument which shows that the list of visual events in Theorem~\ref{thm:surfaces} is indeed exhaustive.
 
 \begin{proof}[Proof sketch for the first part of Theorem~\ref{thm:surfaces}.]
   Let $u \in C_z(X)$ and $L_u := \pi_z^{-1}(u)$.
   Platonova \cite[Main Theorem]{Plat} characterizes all possible \emph{local} singularities in the contour curve.
   In particular, she shows in the case $L_u \in \mathrm{PT}(X)$ that $u$ is not a simple cusp on its branch of the contour curve if and only if $L_u$ is either a flecnodal line or the unique principal tangent at a parabolic point.
   Hence, we only have to characterize the possible multi-local events.
   We argued above that tritangent lines and principal bitangent lines do not yield simple nodes in the contour.
   The final observation is that a line $L_u$ with contact order exactly two at exactly two 
   distinct points of the surface $X$ projects to a simple node $u$ if and only if $L_u$ is 
   not contained in a bitangent plane to $X$ which is tangent at the same two points as $L_u$.
 \end{proof}

\section{Intersection Theory}
\label{sec5}

In this section we derive the degrees of $\mathcal{E}^p(X)$,
$\mathcal{P}^p(X)$, $\mathcal{P}(X)$ and $\mathcal{F}(X)$, and we sketch the
relevant ideas for $\mathcal{E}(X)$, $\mathcal{C}(X)$ and $\mathcal{T}(X)$. We
found Petitjean's proofs in~\cite{Pet} to be lengthy and hard. They require a
full understanding of Colley's multiple point theory~\cite{Col1, Col2}, and
several of the steps are left out. By contrast, the derivations in Salmon's
book \cite{Sal} are inspiring but they lack the rigor of 20th century intersection
theory. 

The exposition that follows refers to the  textbook  by Eisenbud and Harris~\cite{EH}.
We believe that students of that book will find this section to be useful as supplementary reading.
  
We have discussed that the surfaces $\mathcal{E}(X)$ and $\mathcal{P}(X)$ 
are represented by their dual curves in $(\PP^3)^*$. These are the irreducible components in
 the singular locus of the dual surface $X^\vee$.

\begin{proposition}\label{prop:dual_edge_parabolic} The degrees of the curves dual to the edge and parabolic surface of a general surface $X$ of degree $d \geq 4$ or $d \geq 3$, respectively, are
\begin{equation*}
\! {\rm deg}\bigl(\mathcal{E}^p(X)\bigr) =   \frac12 d(d-1)(d-2)(d^3-d^2+d-12)
\quad \hbox{and} \quad
{\rm deg}\bigl(\mathcal{P}^p(X) \bigr) =  4d(d-1)(d-2).
\end{equation*}
\end{proposition}

\begin{proof} 
We  count the bitangent and parabolic planes 
that contain a general point $z \in \PP^3$. Consider the branch curve $C_z(X) \subset \PP^2$.
A bitangent plane to $X$ containing $z$ maps onto a bitangent line of $C_z(X)$.
 A tangent plane to $X$ at a parabolic point that contains $z$ maps onto a flex line of $C_z(X)$.
These correspond, respectively, to the nodes and the cusps of
the dual curve $C_z(X)^\vee $ in $(\PP^2)^\ast$.
We use Pl\"ucker's formula ~\cite[Thm.~1.2.7]{Dol}: for an irreducible plane curve of degree $D$ with
 $\nu$ nodes, $\kappa$ cusps and no other singularities, 
 its dual has degree $D(D-1)-2\nu-3\kappa$.  The genus is
 $\frac12 (D-1)(D-2)-\nu-\kappa$, by the degree-genus formula.

The ramification curve of the projection $\pi_z$ has degree $d(d-1)$ since it is defined by
the equation $f$ of $X$ and its derivative $g(x) := \sum_{i=1}^4 z_i \frac{\partial f}{\partial x_i}$.
Hence $C_z(X)$ has degree $d(d-1)$ as well. The degree
of the dual curve $(C_z(X))^\vee \subset (\PP^2)^\ast$ equals the degree of the dual surface 
$X^\vee$ (see~\cite[proof of Thm.~4.3]{KNT} for details). This common degree is $d(d-1)^2$.

We denote by $\nu_1$ the number of nodes in $C_z(X)$, which is the number of bitangent lines to $X$ passing through the general point $z$.
The number $\kappa_1$ of cusps in $C_z(X)$ is the number of principal tangents to $X$ passing through $z$.
The set of points $x \in X$ that have such a principal tangent through $z$ is the zero
locus of the polynomials $f$, $g$ and $h(x) := \sum_{i=1}^4 z_i \frac{\partial g}{\partial x_i} = z H_f(x) z^T$.
Hence, $\kappa_1 = d(d-1)(d-2)$. Pl\"ucker's formula for $C_z(X)$ tells us that
\begin{align*}
  d(d-1)^2 \,\,=\,\, \deg( (C_z(X))^\vee )\, &
  = \,\,\deg(C_z(X) ) \left( \deg(C_z(X) )-1 \right) - 2 \nu_1 - 3 \kappa_1\\ 
  &= \,\,d(d-1) \left( d(d-1)-1 \right) - 2 \nu_1 - 3 \kappa_1.
\end{align*}
Thus the number of nodes of the branch curve $C_z(X)$ equals $\nu_1 = \frac12 d(d-1)(d-2)(d-3)$.
We can also apply Pl\"ucker's formula to the dual plane curve $C_z(X)^\vee$ to derive
\begin{align}
  \label{eq:pluecker}
  \begin{split}
  d(d-1) \,\, = \,\,\deg( C_z(X) ) \,&= \,\,\deg( (C_z(X) )^\vee ) \left( \deg( (C_z(X) )^\vee )-1 \right) - 2 \nu_2 - 3 \kappa_2\\
  &= \,\,d(d-1)^2 \left( d(d-1)^2-1 \right) - 2 \nu_2 - 3 \kappa_2,
\end{split}
\end{align}
where $\nu_2 = \deg (\mathcal{E}^p(X) )$  and $\kappa_2 = \deg (\mathcal{P}^p(X) )$,
as argued above.
The dual curves $C_z(X)$ and $(C_z(X) )^\vee$ have the same geometric genus.
By the degree-genus formula, it is
\begin{equation}
  \label{eq:genus}
  \frac12 ( d(d{-}1)-1 ) ( d(d{-}1) -2 ) - \nu_1 - \kappa_1 \,\,\,= \,\,\, \frac12 ( d(d{-}1)^2-1 )( d(d{-}1)^2-2 ) - \nu_2 - \kappa_2.
\end{equation}
Solving the equations~\eqref{eq:pluecker} and~\eqref{eq:genus} for
$\nu_2$ and $\kappa_2$ leads to the formulas in Proposition~\ref{prop:dual_edge_parabolic}.
\end{proof}

This proof was entirely classical. By contrast, we derive the degrees of
the visual event surfaces using modern intersection theory \cite{EH}.
At the heart of intersection theory lies the {\em Chow ring}.
For any irreducible projective variety $Y$ and any $c \in \lbrace 0, \ldots, \dim Y \rbrace$, we denote by $\CH^c(Y)$ the free abelian group generated by the set of irreducible closed subvarieties of $Y$ with codimension $c$, modulo \emph{rational equivalence}.
We call elements in $\CH^c(Y)$ \emph{codimension-$c$ cycles}.
The rational equivalence class of a subvariety $Z \subset Y$ is denoted by $[Z]$.
We denote the  class of a point in $Y$ by $\ast \in \CH^{\dim Y}(Y)$, whenever this is well-defined.

\begin{thmDef} {\rm \cite[Thm-Def 1.5]{EH} } 
If $\,Y$ is a smooth projective variety, then there is a unique product structure on
$\,  \CH(Y) = \bigoplus_{c=0}^{\dim Y} \CH^c(Y)\,$ which satisfies the condition:
\begin{center}
If two subvarieties $Z_1,Z_2 \subset Y$ are generically transverse, then $[Z_1][Z_2] = [Z_1 \cap Z_2]$.
\end{center}
By generically transverse we mean that every irreducible component of $Z_1 \cap Z_2$ contains a point where $Z_1$ and $Z_2$ are transverse.
This makes $\CH(Y)$ into an associative commutative graded ring,  called the \emph{Chow ring} 
of~$Y$.
\end{thmDef}

\begin{example} \rm \cite[Theorem 2.1]{EH} 
\label{ex:chowRingProjSpace} 
The Chow ring of projective $3$-space is $\CH(\PP^3)= \ZZ[h] /  \langle h^4 \rangle$, where $h \in \CH^1(\PP^3)$ is the (rational equivalence) class of a plane in $\PP^3$.
\hfill $\diamondsuit$
\end{example}

\begin{example} \rm \cite[Theorem 3.10]{EH} \
The Chow ring of the Grassmannian ${\rm Gr}(1,\PP^3)$ is a graded 
$\ZZ$-algebra with Hilbert function $(1,1,2,1,1)$. The graded components are
the free $\ZZ$-modules
\begin{align*}
\CH^0(\Gr) = \ZZ \cdot [\Gr], \;\; 
& \CH^1(\Gr) &= \,\,\ZZ \gamma_1, \,\,\, \;\; 
& \CH^2(\Gr) &= \,\ZZ \gamma_1^2 \oplus \ZZ \gamma_2, \\
& \CH^3(\Gr) &= \ZZ \gamma_1\gamma_2, \;\;
& \CH^4(\Gr) &= \,\, \ZZ \ast ,\,\,\,
\end{align*}
where $\gamma_1$ represents the threefold of all lines that meet a given line,
and $\gamma_2$ represents the surface of all lines that lie in a given plane.
The multiplicative structure is given by
\begin{align*}
  \gamma_1^3 = 2 \gamma_1 \gamma_2\,, \quad
  \gamma_2^2 = \ast\,, \quad
  \gamma_1^2\gamma_2 = \ast\,, \quad
  \gamma_1^4 = 2 \ast.
\end{align*}
Further, the class  $\gamma_1^2-\gamma_2$ represents
the surface of all lines passing through a given point.
\hfill $\diamondsuit$
\end{example}

\begin{example} \rm \cite[Proposition 9.10]{EH} The {\em universal line}
$\,\Phi= \lbrace (x,L) \in \PP^3 \times \Gr \mid x \in L \rbrace$ 
 is a $5$-dimensional
smooth projective variety. The Chow ring of this variety equals
$\CH(\Phi) = \CH(\Gr)[H] /  \langle H^2 - \gamma_1 H + \gamma_2 \rangle$.
Here, the class $H \in \CH^1(\Phi)$ represents the preimage of a plane in $\PP^3$ under the projection $\Phi \to \PP^3$ onto the first factor.
\hfill $\diamondsuit$
\end{example}

\begin{proposition}
  \label{prop:flecDegree}
  For a general surface $X$ of degree $d \geq 4$, the degrees of the flecnodal surface $\mathcal F(X)$ and the flecnodal curve $F$ on $X$ are
\begin{equation*}
  \! {\rm deg}\bigl(\mathcal{F}(X)\bigr) =   2d(d-3)(3d-2)
\quad \hbox{and} \quad
{\rm deg}\bigl(F \bigr) =  d(11d-24).
\end{equation*}
\end{proposition}

\begin{proof}
  Let $X^4_\Phi \subset \Phi$ 
  be the incidence variety of pairs $(x,L)$ such that the line $L$ has contact of order at least $4$ at the point $x \in X$.
The degree of the flecnodal surface
 $\mathcal{F}(X)$ is the number of intersections of $\mathcal{F}(X)$ with a general line.
This is the number of pairs $(x,L) \in X^4_\Phi$ such that $L$ meets a general line. 
In particular, we have $\,[X^4_\Phi] \cdot \gamma_1 = \deg(\mathcal{F}(X)) \ast \,$ in $ \CH(\Phi)$.

We shall compute $[X^4_\Phi] \in \CH^4(\Phi)$ via Chern classes.
The \emph{top Chern class} $c_r(\mathcal{E}) \in \CH^r(Y)$ of a vector bundle $
\mathcal{E}$ of rank $r$ on a smooth 
variety $Y$ with a global section $\sigma: Y \to \mathcal{E}$ is the class of the vanishing locus of $\sigma$. This definition is independent of the chosen global section. 
Fix any vector bundle $\mathcal{E}$ on $\Phi$ and any integer $m \in \mathbb{N}$.
By \cite[Theorem~11.2]{EH}, 
there is a new vector bundle $\mathcal{J}_{\Phi / \Gr}^m(\mathcal{E})$ on $\Phi$ whose
 fiber at $(x,L)$  is the space of all germs of sections of $\mathcal{E}|_{\lbrace (y,L) \in \Phi \rbrace}$ 
  modulo those that vanish to order $ \geq m+1$ at $(x,L)$.
  This is called the \emph{bundle of relative principal parts} or the \emph{relative jet bundle}. 
We shall compute its top Chern class.

Let now $\mathcal{E}$ be the pullback to $\Phi$ of the line bundle $\mathcal{O}_{\PP^3}(d)$.
A global section of $\mathcal{E}$ is given by the homogeneous polynomial 
in $x_1,x_2,x_3,x_4$ of degree $d$ that defines $X$ in  $\PP^3$.
Restricting this polynomial to the line $L$ gives a global 
section of $\mathcal{E}|_{\lbrace (y,L) \in \Phi \rbrace}$. 
That global section vanishes at $(x,L)$ if and only if $L$ has contact of order at least $ m+1$ at $x \in X$.
Hence, the top Chern class of $\mathcal{J}_{\Phi / \Gr}^m(\mathcal{E})$ 
 is the class of the subvariety of all pairs $(x,L) $
 in $\Phi$ such that $L$ has contact of order at least $ m+1$ at $x \in X$.
 In particular, $c_4(\mathcal{J}_{\Phi / \Gr}^3( \mathcal{E})) = [X^4_\Phi]$.

 In addition, we see from Theorem~11.2 in~\cite{EH} that $\mathcal{J}_{\Phi / \Gr}^m(\mathcal{E})$ 
agrees locally with
\begin{equation}
\label{eq:bundlesum}
\mathcal{E} \oplus \left( \mathcal{E} \otimes  \Omega_{\Phi / \Gr} \right) \,\oplus \,
\left( \mathcal{E} \otimes   \Sym^2 \Omega_{\Phi / \Gr} \right) \,\oplus \, \cdots \,\oplus \,
\left( \mathcal{E} \otimes  \Sym^m \Omega_{\Phi / \Gr} \right),
   \end{equation}
where $\Omega_{\Phi / \Gr}$ is the \emph{relative cotangent bundle}, which has rank 1 in our case.
We compute the top Chern class of $\mathcal{J}_{\Phi / \Gr}^m(\mathcal{E})$ from
its representation (\ref{eq:bundlesum}).
From~\cite[p.~395]{EH} 
we have $c_1 (\Sym^m \Omega_{\Phi / \Gr}) = m(\gamma_1 - 2H) $ in $ \CH^1(\Phi)$.
Since the equation $f$  of $X$ gives a global section of $\mathcal{E}$, we further have
  $c_1(\mathcal{E}) = d H$ in $ \CH^1(\Phi)$. The top Chern 
class of the tensor product of two line bundles is the sum of their top Chern classes 
\cite[Prop.~5.17]{EH}. Hence
$$\,c_1(\mathcal{E} \otimes \Sym^i \Omega_{\Phi / \Gr}) \,\,= \,\,
(d-2i)H+i\gamma_1 \qquad \hbox{in} \quad 
\CH^1(\Phi) \quad \hbox{for} \quad i = 0,1,2,\ldots,m.  $$
Finally, \emph{Whitney's sum formula}~\cite[Theorem~5.3]{EH} says that the top 
Chern class of a direct sum of vector bundles is the product of the top Chern classes of the summands. 
This implies
 \begin{align*}
   c_{m+1}(\mathcal{J}_{\Phi / \Gr}^m(\mathcal{E})) \,\,\, = \,\,\,\,
 \prod_{i=0}^m \left( (d-2i)H+i\gamma_1 \right) 
 \qquad \hbox{in} \quad  \CH^{m+1}(\Phi).
\end{align*} 
We have $H^4=0$, by Example~\ref{ex:chowRingProjSpace}
and $H^3 \gamma_1^2 = \ast$,
since exactly one line meets two general lines and passes through a given point.
Finally, $H^2 \gamma_1^3 = 2 \ast$ and $H \gamma_1^4 = 2 \ast$, since
 exactly two lines  meet four general lines in $\PP^3$.  Putting these pieces together, we 
 get the desired formula
$$ \begin{matrix}
\deg(\mathcal{F}(X)) \ast &\!=& [X^4_\Phi] \cdot \gamma_1
\,\,=\,\, c_4(\mathcal{J}_{\Phi / \Gr}^3(\mathcal{E})) \cdot \gamma_1
\,\,\,= \,\,\, \gamma_1 \prod_{i=0}^3 \left( (d-2i)H+i\gamma_1 \right) \\
&\! =& \!\! \left( 6d^3{-}44d^2{+}72d \right) H^3\gamma_1^2 + \left( 11d^2{-}36d \right) H^2 \gamma_1^3 + 6d H \gamma_1^4  \,= \, 2d(d-3)(3d-2)\! \ast. 
\end{matrix}
$$
An analogous computation yields the degree of the flecnodal curve $F$ on the surface  $X$. We find
that $\, \deg(F) \ast \,\,=\,\, [X^4_\Phi] \cdot H \,\,=\,\, H \prod_{i=0}^3 \left( (d-2i)H+i\gamma_1 \right) 
\,\,=\,\, d(11d-24) \ast.  $
\end{proof}

\begin{proposition}
  \label{prop:parDegree}
  The degree of the parabolic surface of a general surface $X$ of degree $d \geq 3$~is
\begin{equation*}
  \! {\rm deg}\bigl(\mathcal{P}(X)\bigr) \,\,= \,\,  2d(d-2)(3d-4).
\end{equation*}
\end{proposition}

\begin{proof}
We consider the incidence variety $P^3_\Phi$ that consists of all pairs $(x,L)$ in $\Phi$
with the property that $x$ is parabolic on $X$ and $L$ has contact order at least $
3$ at $x \in X$. Set-theoretically, this is the intersection of the variety
$X_\Phi^3 \subset \Phi$ of pairs $(x,L)$ such that $L$ is a principal tangent
to $X$ at $x$ with the variety $P^1_\Phi \subset \Phi$ of pairs $(x,L)$ such
that $x$ is parabolic on $X$. Since $\codim_\Phi P^3_\Phi = 4 < 3 + 2 =
\codim_\Phi X_\Phi^3 + \codim_\Phi P^1_\Phi$, we cannot simply multiply their
classes in $\CH(\Phi)$ to get the class of $P^3_\Phi$.

We shall compute all relevant classes in $\CH(\Phi_X)$, where $\Phi_X $ is the
variety of  pairs $(x,L) \in \Phi$ with $x \in X$. The varieties $X_\Phi^3$
and $P_\Phi^1$ intersect with the expected codimension in $\Phi_X$, although
this intersection is not transverse. It is difficult to describe the Chow ring
of $\Phi_X$. We find some generators and relations of $\CH(\Phi_X)$ by taking
pullbacks of elements in $\CH(\Phi)$ under the inclusion $\Phi_X
\hookrightarrow \Phi$. 
Recall that
the pullback of a class $[Z]$ under a nice
enough morphism $f$ between varieties is $[f^{-1}(Z)]$,
by~\cite[Thm.~1.23]{EH}. We denote by $E \in \CH^1(\Phi_X)$ the pullback of
the hyperplane class $H \in \CH^1(\Phi)$ and by $\Gamma_i \in \CH(\Phi_X)$ the
pullback of $\gamma_i \in \CH(\Phi)$ for $i \in \lbrace 1,2 \rbrace$.

By definition, $P^1_\Phi$ is the preimage of the parabolic curve $P \subset X$ under the map $\Phi_X \to X$.
The parabolic curve $P$ is the intersection of $X$ with the Hessian surface of degree $4(d-2)$. 
Thus, the class of $P$ in $\CH(X)$ is $4(d-2)e$, where $e \in \CH^1(X)$ is the pullback of the hyperplane class $h \in \CH^1(\PP^3)$ under the inclusion $X \hookrightarrow \PP^3$.
Since $E \in \CH^1(\Phi_X)$ is also the pullback of $e \in \CH^1(X)$ under the projection $\Phi_X \to X$, we have $[P^1_\Phi] = 4(d-2)E \in \CH^1(\Phi_X)$.
We compute $[X_\Phi^3]$ in the Chow ring of $\Phi$ using
the formula (\ref{eq:bundlesum}) with $m=2$:
\begin{align*}
  [X_\Phi^3] &= c_3 (\mathcal{J}^2_{\Phi / \Gr}(\mathcal{E})) 
= \prod_{i=0}^2 ( (d\!-\!2i) H \!  +\! i \gamma_1 ) 
= dH \left( (d^2\!-\!6d\!+\!8) H^2 + (3d\!-\!8)H \gamma_1 + 2 \gamma_1^2  \right).
\end{align*}
The right hand side lives in the Chow ring of the $5$-dimensional variety $\Phi$
and we pull it back to the Chow ring of its $4$-dimensional subvariety $\Phi_X$.
That pullback is the codimension-$3$ cycle
$dE \beta$, where $\beta := (d^2-6d+8) E^2 + (3d-8)E \Gamma_1 + 2 \Gamma_1^2 \in \CH^2(\Phi_X)$.
Since pullback preserves codimension, 
$[X_\Phi^3] \in \CH^2(\Phi_X)$ cannot be equal to $dE \beta \in \CH^3(\Phi_X)$.
Instead~\cite[Theorem~13.7]{EH} tells us that $d E \beta = [X_\Phi^3] \cdot c_1 (\mathcal{N}_{\Phi_X / \Phi}) \in \CH^3(\Phi_X)$, where $\mathcal{N}_{\Phi_X / \Phi}$ is the normal bundle of $\Phi_X$ in $\Phi$.
By~\cite[Prop.-Def.~6.15]{EH}, we have  $c_1 (\mathcal{N}_{\Phi_X / \Phi}) = dE$ and $d E \beta =d E  [X_\Phi^3]$.

We cannot yet say that $[X_\Phi^3] = \beta$ in $\CH^2(\Phi_X)$ because
multiplication with $E$ has non-trivial kernel. However, since $E$ is a factor of 
$[P^1_\Phi]$, we conclude that
$[X_\Phi^3][P^1_\Phi] = 4(d-2)E \beta$ 
in the Chow ring of $\Phi_X$.
Since a general point on $X$ has two principal tangent lines, the intersection multiplicity of the varieties $P^1_\Phi$ and $X_\Phi^3$ is 2.
Therefore, by~\cite[Theorem~1.26]{EH}, we have  $[P^3_\Phi] = \frac12 [X_\Phi^3][P^1_\Phi] = 2(d-2)E\beta$ in $\CH(\Phi_X)$.
As in the case of flecnodal lines, the degree of the parabolic surface $\mathcal{P}(X)$ is the 
number of points in the $0$-dimensional cycle
\begin{align*}
  \,\,[P^3_\Phi] \cdot \Gamma_1
  \,\,= \,\, 2(d-2) \left( (d^2-6d+8) E^3 \Gamma_1 + (3d-8)E^2 \Gamma_1^2 + 2 E\Gamma_1^3 \right).
\end{align*}
Finally, we use the pushforward of the inclusion $i: \Phi_X \hookrightarrow \Phi$ to express the above monomials in the point class $\ast$ of $\CH(\Phi)$.
The pushforward $f_\ast$ of a proper morphism $f$ maps $[Z]$ to 0 if $\dim(f(Z)) < \dim(Z)$, and otherwise to $\nu [f(Z)]$ where $\nu \in \mathbb{Z}$ denotes the degree of the restricted map $f|_Z: Z \to f(Z)$.
Using the \emph{push-pull formula}~\cite[Theorem~1.23]{EH}, we derive $i_\ast(E^2 \Gamma_1^2) = i_\ast( i^\ast( H^2 \gamma_1^2) \cdot [\Phi_X] )  = H^2 \gamma_1^2 \cdot i_\ast([\Phi_X]) = dH^3 \gamma_1^2 = d \ast$.
Similarly, we get $i_\ast (E \Gamma_1^3) = 2d \ast$ and $i_\ast(E^3 \Gamma_1) = 0$.
Hence, $\deg(\mathcal{P}(X)) \ast = i_\ast([P^3_\Phi] \cdot \Gamma_1) =  2d(d-2)(3d-4) \ast$.
\end{proof}

\begin{proposition}
  \label{prop:edgeDegree}
  The degree of the edge surface of a general surface $X$ of degree $d \geq 4$ is
\begin{equation*}
  \! {\rm deg}\bigl(\mathcal{E}(X)\bigr) \,\,=\,\,   d(d-2)(d-3)(d^2+2d-4).
\end{equation*}
\end{proposition}

\begin{proof}[Proof sketch.]
  We describe an idea for computing the degree of $\mathcal{E}(X)$ with a mix of classical methods and intersection theory, similar to Petitjean's approach~\cite[Prop.~4.14-4.16]{Pet}.
The Gauss map $\gamma: X \to (\PP^3)^\ast$ assigns to each point $x \in X$ the tangent plane to $X$ at $x$.
The preimage of $\mathcal{E}^p(X)$ under $\gamma$ is the node-couple curve $C \subset X$,
and the restriction $\gamma |_C: C \to \mathcal{E}^p(X)$ is a 2-to-1 covering of $\mathcal{E}^p(X)$.
From this we get that $\gamma_\ast([C]) = 2 [\mathcal{E}^p(X)]$ in $\CH( (\PP^3)^\ast )$. 

We can compute the class $[C]$ in $\CH^1(X)$ as follows.
Let us denote by $h^\ast \in \CH^1( (\PP^3)^\ast )$ the class of a hyperplane in $(\PP^3)^\ast$, and by $e^\ast \in \CH^1(X)$ the pullback of $h^\ast$ under the Gauss map $\gamma$.
Since $e \in \CH^1(X)$ (which we defined as the pullback of the hyperplane class $h \in \CH^1(\PP^3)$ under the inclusion $j:X \hookrightarrow \PP^3$) generates $\CH^1(X)$, we know that $e^\ast = \alpha e$ for some $\alpha \in \ZZ$. 
The push-pull formula implies 
\begin{align*}
  \gamma_\ast( (e^\ast)^2 ) &= (h^\ast)^2 \gamma_\ast( [X] ) = (h^\ast)^2  \cdot d(d-1)^2h^\ast = d(d-1)^2 \ast, \\
  j_\ast( (e^\ast)^2) &= \alpha^2 j_\ast(e^2) = \alpha^2 h^2 j_\ast([X]) = \alpha^2 h^2 \cdot dh = \alpha^2 d \ast,
\end{align*}
and thus $\alpha = d-1$.
Hence, we have $(d-1) \gamma_\ast (e) = \gamma_\ast(e^\ast) = d(d-1)^2 (h^\ast)^2$
and $\gamma_\ast(e) = d(d-1) (h^\ast)^2$.
Writing $[C] = \delta e$ for $\delta \in \ZZ$, we finally derive
\begin{align*}
  2 \deg( \mathcal{E}^p(X) ) (h^\ast)^2 = 2 [\mathcal{E}^p(X)] = \gamma_\ast([C]) = \delta \gamma_\ast(e) = \delta d(d-1) (h^\ast)^2
\end{align*}
and $\delta = (d-2)(d^3-d^2+d-12)$ by Proposition~\ref{prop:dual_edge_parabolic}.

The numbers $a$ and $b$ of cusps and nodes of the curve $C$ are given by \cite[Prop.~4.15]{Pet}:
\begin{align*}
  a &= 4d(d-2)(d-3)(d^3+3d-16), \\
  b &= \frac12 d(d-2)(d^7-4d^6+7d^5-45d^4+114d^3-111d^2+548d-960).
\end{align*}
Since the curve $C$ does not have any other singularities, we can apply the intersection theoretic genus formula:
the geometric genus of $C$ is given by the number of points in the 0-dimensional cycle $\frac12([C]^2 + K_X [C])$ plus the number  $1 -a -b$~(cf. \cite[Sec.~2.4.6]{EH}), where $K_X = (d-4) e$ is the \emph{canonical class} (cf.~\cite[Sec.~1.4.3]{EH}):
\begin{align*}
  j_\ast([C]^2+K_X[C]) &= \left(\delta^2+\delta(d-4) \right) j_\ast(e^2) = \left(\delta^2+\delta(d-4) \right) d \ast \quad\quad \text{and} \\
  \mathrm{genus}(C) &= \frac12 \left( \delta^2+\delta(d-4) \right) d +1 -a-b \\ 
  &= 3d^6-15d^5+27d^4-104d^3+340d^2-336d+1.
\end{align*}
The map $\gamma': C' \to \mathcal{E}^p(X)'$ between the normalizations of $C$ and $\mathcal{E}^p(X)$ is exactly ramified at the godrons, i.e., the points of tangency of the parabolic curve $P$ and the flecnodal curve $F$ on $X$ (see~\cite[pp.~229-231]{Pie3}). 
Thus, the number of those points is the number of points in the 0-dimensional cycle $\frac12 [P][F] = \frac12 \cdot 4(d-2)e \cdot (11d-24) e$, which is $\frac12 j_\ast( [P][F]) = 2(d-2)(11d-24) j_\ast(e^2) = 2d(d-2)(11d-24) \ast$.
We find the genus of $\mathcal{E}^p(X)$ by applying the Riemann-Hurwitz formula~\cite[Cor.~2.4]{Har} to $\gamma'$:
\begin{align*}
  2\,\mathrm{genus}(C) -2 \,\,=\,\, 2 (2 \, \mathrm{genus}( \mathcal{E}^p(X)) -2 ) + 2d(d-2)(11d-24).
\end{align*}
Hence, $\,{\rm genus}(\mathcal{E}^p(X)) = \frac12(3d^6-15d^5+27d^4-115d^3+386d^2-384d+2)$.
Finally, the degree of the surface $\mathcal{E}(X) = (\mathcal{E}^p(X) )^\vee$ is the degree of the tangential surface of
the curve  $\mathcal{E}^p(X)$. The latter degree equals $\,2\bigl( \deg(\mathcal{E}^p(X) ) + \mathrm{genus} (\mathcal{E}^p(X) ) {-}1 \bigr)
-a = d(d-2)(d-3)(d^2+2d-4)$ (see~\cite[Thm.~3.2]{Pie1}), since $a$ is also the number of cusps on $\mathcal{E}^p(X)$.
 \end{proof}

Proving the degrees of $\mathcal{C}(X)$ and $\mathcal{T}(X)$ is more technical. We will not include this here.
One method is Colley's multiple point theory~\cite{Col1, Col2}. Alternatively, one can write
 $\mathcal{C}(X)$ and $\mathcal{T}(X)$ as the intersection of loci of (principal) tangents in the fiber product $\Phi \times_{\Gr} \Phi$ or $\Phi \times_{\Gr} \Phi \times_{\Gr} \Phi$ and 
 remove extra components in the intersection by blowing these~up.

\section{Computing Visual Events}
\label{sec6}

Theorem~\ref{thm:surfaces} gives the degrees of the irreducible
components of the visual event surface when
$X$ is a general surface of degree $d$ in $\PP^3$.
Table \ref{tab:surfaces} below summarizes these
degrees for $d \leq 7$.
One notices that the degrees are now much larger than those
for curves in Table~\ref{tab:curves}.

\begin{table}[h]
\centering
\begin{tabular}{ccccccc}
  \Xhline{2\arrayrulewidth}
$d$ && $ {\rm deg}(\mathcal{F}(X)) $& 
$ {\rm deg}(\mathcal{C}(X)) $& 
$ {\rm deg}(\mathcal{T}(X)) $&
$  {\rm deg}(\mathcal{E}(X))$ &
$ {\rm deg}(\mathcal{P}(X)) $\\
\hline
 3 && 0 & 0 &  0 & 0 & 30 \\
4 && 80 & 0 & 0 & 160 & 128 \\
5 && 260 & 510 & 0 & 930 & 330 \\
6 && 576 & 2448 & 624 & 3168 & 672 \\
7 && 1064 & 7308 & 3808 & 8260 & 1190 \\
\Xhline{2\arrayrulewidth}
 \end{tabular}
\caption{\label{tab:surfaces}  
Degrees of the components of the
visual event surface of a general surface}
\end{table}

The degrees in Table \ref{tab:surfaces}
pose a challenge because a homogeneous polynomial
in four unknowns of degree $\delta$ 
can have as many as $\binom{\delta+3}{3}$ terms.
For instance, if $X$ is a quintic surface then its flecnodal surface $\mathcal{F}(X)$ has
degree $\delta = 260$, so the expected number of terms is
$\binom{\delta+3}{3} = 2997411$. 
In this section we address this challenge. See Example~\ref{ex:260victory} for a solution.

Throughout this section, we make use of the multiple root loci  for binary
forms. The ideals of these varieties are defined by homogeneous  polynomials
in the coefficients $c_0,c_1,\ldots c_d$ of
\begin{equation}
\label{eq:binaryform}
 c_0 t^d + c_1 t^{d-1}  + c_2 t^{d-2}  + \cdots +
  c_{d-1} t  + c_d .
  \end{equation} 
For a partition $\lambda = (\lambda_1,\ldots,\lambda_k) \in \mathbb{N}^k$ with $\sum_{i=1}^k \lambda_i
\le d$, we write $\Delta_\lambda(d)$ for the homogeneous prime ideal in
$\RR[c_0,\ldots,c_d]$ whose variety consists of polynomials \eqref{eq:binaryform} 
that have $k$ complex roots with multiplicities $\lambda_1,\ldots,\lambda_k$. 
The varieties are called {\em multiple root loci} in \cite{LS}.
For example, $\Delta_{(4)}(d)$ is the prime ideal for polynomials of degree $d$ with one quadruple root.

\begin{example} \label{ex:viervier} \rm
Let $d=4$ in (\ref{eq:binaryform}) and consider univariate quartics
that have a single root of multiplicity four. These quartics are the points on a rational 
normal curve in $\PP^4$. The prime ideal of this curve is $\Delta_{(4)}(4)$. It
 is  generated by the six $2 {\times} 2$-minors of the $2 {\times} 4$-matrix
\begin{equation*}
\begin{pmatrix}
12 c_0 & 3 c_1 & 2 c_2 & 3 c_3 \\
  3 c_1 & 2 c_2 & 3 c_3 & 12 c_4
\end{pmatrix}.
\end{equation*}
The variety of quintics $(d=5)$ with one root of multiplicity four
is the tangential surface of the rational normal curve in $\PP^5$.
Its ideal is the complete intersection of three quadrics:
\begin{equation}
\label{eq:l^4b}
\Delta_{(4)}(5) \,\,\, = \,\,\,
\big\langle\,
     20 c_0 c_4-8 c_1 c_3+3 c_2^2 \,,\,\,
50 c_0 c_5-6 c_1 c_4+c_2 c_3 \,,\,\,
20 c_1 c_5 - 8 c_2 c_4+3 c_3^2 \,
     \big\rangle.
\end{equation}

Another multiple root locus was seen in Example \ref{ex:binonical3}.
The ideal $\Delta_{(2,2)}(5)$ is minimally generated by
$10$ quintics. We used this to compute 
the edge surface of a degree $5$ curve.
\hfill $\diamondsuit$
\end{example}

We refer to \cite[Table 1]{LS} for details on the ideals
$\Delta_\lambda(d)$. Some relevant instances are
listed in Table \ref{tab:hwangrae}. Its entries 
are copied from \cite[Table 1]{LS}. For instance, the entry $\,6^{10},
8^{38}\,$ in the last column means that $\Delta_{(3,2)}(7) $  is minimally
generated by $10$ sextics and $38$ octics.

The ideals $\Delta_\lambda(d)$ can be computed either by direct
implicitization, or by  using {\em subresultants}~\cite{ADG}. The $i$-th
subresultant $S_i(h_1,h_2)$ of two polynomials $h_1(t)$ and $h_2(t)$ is a
polynomial of degree at most $i$ whose coefficients are the determinants of
particular minors of the Sylvester matrix of $h_1$ and $h_2$. The vanishing of
$S_i(h_1,h_2)$ for $0 \le i \le d-1$ means that the greatest common divisor
(GCD) of $h_1$ and $h_2$ has degree at least $d$. Moreover, if $S_d(h_1,h_2)$
is not zero, it is exactly this GCD. If we let $h_d$ be the polynomial
\eqref{eq:binaryform} and $h_d'$ be its derivative with respect to $t$, then
the condition that $h_d$ has roots with multiplicity $\lambda =
(\lambda_1,\ldots,\lambda_k)$ is equivalent to the fact that the GCD of $h_d$
and $h_d'$ has degree $\sum_{i=1}^k (\lambda_i - 1)$ and has roots with
multiplicities $\lambda' = (\lambda_1-1,\ldots,\lambda_k-1)$. This allows us
to compute the ideal $\Delta_\lambda(d)$ recursively.


\begin{table}[h]
  \centering \begin{tabular}{cccccccccc}
    \Xhline{2\arrayrulewidth}
\hbox{Ruled surface} & {Partition}  &&   $d=4$ && $d=5$ && $d=6$ && $d=7$ \\
\hline
\rule{0pt}{2.2ex}
$\mathcal{F}(X)$ & $\lambda = (4)$ &&      $2^6$    &&    $2^3$   && $2^1,3^3,4^1$  && $ 4^{20}$           \\
$\mathcal{C}(X)$ & $\lambda = (3,2)$ &&        &&      $4^{28}$ &&  $4^1,5^3,6^{31} $   && $ 6^{10}, 8^{38} $           \\
$\mathcal{T}(X)$ & $\lambda = (2,2,2)$ &&         &&        &&   $ 4^{45}$    && $ 6^{78}$   \\
\Xhline{2\arrayrulewidth}
\end{tabular}
\caption{\label{tab:hwangrae}  
The ideals $\Delta_{(\lambda)}(d)$ of multiple root loci relevant for visual events of surfaces}
\end{table}

In what follows we assume that the ideals $\Delta_\lambda(d)$ have been 
pre-computed for $d \le 7$. We use these data to compute the
curves $\mathcal F^\ell(X)$, ${\mathcal C}^\ell(X)$ and $\mathcal{T}^\ell(X)$
in the Grassmannian ${\rm Gr}(1,\PP^3)$. The correspondence between the three
multi-local events $\mathcal{F},\mathcal{C},\mathcal{T}$ and the three special
partitions $\lambda$ was seen on the right side in Figure \ref{fig:surface-graph}, where $\lambda$ was denoted by~$m$.

Let $f = f(x_1,x_2,x_3,x_4)$ be the polynomial of degree $d$ that defines the
surface $X$. We parameterize the line in $\PP^3$ with Pl\"ucker coordinates
$q$ using a parameter $t$. For instance, we can write (\ref{eq:lineparaone})
dually as $\,z(t) =  (q_{12}: tq_{12}: tq_{13} -  q_{23}: tq_{14} - q_{24})$.
We substitute $z(t)$ into the polynomial $f$, and we regard $f(z(t))$ as a
univariate polynomial in $t$, written as in (\ref{eq:binaryform}). The
coefficients $c_i$ are now homogeneous expressions of degree $d$ in the
Pl\"ucker coordinates $q$. At this point, we  substitute these expressions
$c_i(q)$ into the generators of $\Delta_{\lambda}(d)$. The result is an ideal
in the Pl\"ucker coordinates $q$ that defines the desired curve set-theoretically. 
The same method can be applied when local coordinates on the
Grassmannian ${\rm Gr}(1,\PP^3)$ are preferred. In this case, we parameterize
the line in $\PP^3$ by $z(t) = (1 : t : \alpha + t \gamma : \beta + t \delta)$.

\begin{example} \label{ex:260victory}
\rm Let $d=5$ and consider the smooth quintic surface $X$ defined by
\[ f \,\,=\,\, x_1^5+ x_2^5 + x_3^5 + x_4^5 + (x_1 + x_2 + x_3 + x_4)^5 + x_1 x_2 x_3
x_4(x_1 + x_2 + x_3 + x_4). \] 
We compute the curve $\mathcal F^\ell(X)$ in
${\rm Gr}(1,\PP^3)$ that represents the flecnodal surface.
Its prime ideal has degree $260$ and is generated by $10$ sextics plus the
Pl\"ucker quadric~(\ref{eq:pluckerrel}). This computation was done with the
method above, starting from the ideal $\Delta_{(4)}(5)$  in
\eqref{eq:l^4b}.
\hfill $\diamondsuit$
\end{example}

Let us shift gears and focus on the local events $\mathcal{P}$ and
$\mathcal{E}$, seen on the left in Figure \ref{fig:surface-graph}. We start
with the parabolic surface $\mathcal{P}(X)$. Let $X$ be defined by a
polynomial $f \in \RR[x_1,x_2,x_3,x_4]$. The ideal $I(P)$ of the parabolic
curve $P$ is defined by $f$ and the determinant of the Hessian matrix $H_f$.
Consider the incidence variety of the parabolic curve and its tangent planes,
that is $\lbrace (x,T_x(X)) \mid x \in P \rbrace \,\subset \,\PP^3 \times
(\PP^3)^*$. We compute the ideal of the incidence variety by adding the $2
\times 2$-minors of the matrix $\left(\begin{smallmatrix}  \partial{f}/\partial{x_1} &
 \partial{f}/\partial{x_2} &  \partial{f}/\partial{x_3} &  \partial{f}/\partial{x_4} \\
y_1 & y_2 & y_3 & y_4 \end{smallmatrix} \right)$ to $I(P)$. We then
saturate the resulting ideal by $\langle x_1,x_2,x_3,x_4 \rangle$ 
and afterwards eliminate $x_1,x_2,x_3,x_4$.
This furnishes the ideal of the dual curve
$\mathcal{P}^p(X)$ in   $(\PP^3)^*$, which  encodes the developable surface $\mathcal{P}(X)$.

\begin{proposition}
\label{ref:quarticsextic} If $X$ is a general cubic surface, the curve
$\mathcal{P}^p(X)$ is a complete intersection of a quartic and a sextic,
obtained from the two basic invariants of ternary cubics.
\end{proposition}

\begin{proof}
A classical fact from invariant theory states that the ring of invariants for ternary cubics
is generated by a quartic and a sextic, and these vanish precisely when the cubic has a cusp.
We represent $X$ as the blow-up of $\PP^2$ at six points, namely as the image
of  the map to $\PP^3$ defined by four independent cubics $f_1,f_2,f_3,f_4$ in $x,y,z$
that vanish at these points. We now consider the cubic
$y_1 f_1 + y_2 f_2 + y_3 f_3 + y_4 f_4$, where $y_1,y_2,y_3,y_4$ are unknowns.
Plugging this cubic into the two basic invariants gives the condition for a plane to
meet $X$ in a cuspidal curve. Hence that locus in $(\PP^3)^*$ is the complete intersection of a quartic and a sextic.
\end{proof}

For a general parabolic point $x$ of $X$, the Hessian matrix $H_f(x)$ has rank
three. Its kernel represents a unique point $p_x$ in $\PP^3$.  We use
the following simple fact to compute $\mathcal{P}^\ell(X)$.

\begin{lemma} \label{lem:thepointpx}
For $x \in P$, the point $p_x$ lies on the unique principal tangent of $\,X$ at $x$. 
\end{lemma}

\begin{proof}
The relation $x\, H_f(x) \, p_x^T = 0$ holds. Euler's relation shows that
$\,x \,H_f(x)$ is the gradient vector of $f$ at $x$. Hence $p_x$  lies on the tangent plane to $X$ at $x$.
Furthermore,  $p_x$ belongs to the principal tangent since
$\,p_x \, H_f(x)\, p_x^T $ is zero. 
Hence $x$ and $p_x$ span the principal tangent.
\end{proof}

The curve $\mathcal{P}^\ell(X) \subset {\rm Gr}(1,\PP^3)$ is the
collection of the lines spanned by a general parabolic point $x$ and the
corresponding point $p_x$ from Lemma~\ref{lem:thepointpx}. This allows us to
compute the ideal of $\mathcal{P}^\ell(X)$ in dual Pl\"ucker coordinates
$q_{12},q_{13},\ldots,q_{34}$. First, we recover the ideal ${I}$ of the
incidence variety $\lbrace (x,y) \,|\, x \in P, \,y \in \ker H_f(x)
\rbrace$ by adding the four entries of the column vector $\,H_f(x) \cdot y\,$
to the ideal $I(P) = \langle f, \det H_f(x) \rangle$. Secondly, we consider
the map from the coordinate ring of the Grassmannian to the quotient ring of
${I}$ that maps  Pl\"ucker coordinates $q_{ij}$ to the $2 \times 2$-minors of
$\left( \begin{smallmatrix}  x_1 & x_2 & x_3 & x_4\\ y_1 & y_2 & y_3 & y_4
\end{smallmatrix} \right)$. The kernel of this ring map is the ideal of the
curve $\mathcal{P}^\ell(X) \subset {\rm Gr}(1,\PP^3)$. This ideal is generated
by $4$~cubics and $6$ quartics, plus the Pl\"ucker quadric, in the unknowns
$q_{12},\ldots,q_{34}$. One verifies computationally that the ideal defines a
curve of degree $30$ in $\PP^5$ for a cubic surface $X$. Of course, this curve
is $\mathcal{P}^\ell(X)$.

\begin{example}
\rm Let $d=3$ and consider the Fermat cubic $X$ defined by  $\,f \,=\, x_1^3 + x_2^3 + x_3^3 + x_4^3$.
We can easily compute the ideal of the curve $\mathcal{P}^\ell(X)$ as described above,
and from this we find  the parabolic surface $\mathcal{P}(X)$.
It decomposes into irreducible components of low degree:
\begin{align*}
&(x_0+x_1) \cdot (x_0+x_2) \cdot(x_0+x_3) \cdot(x_1+x_2) \cdot(x_1+x_3) \cdot(x_2+x_3) \\
\cdot 
&(x_0^2-x_0x_1+x_1^2) \cdot (x_0^2-x_0x_2+x_2^2) \cdot (x_0^2-x_0x_3+x_3^2)\\
 \cdot
&(x_1^2-x_1x_2+x_2^2) \cdot (x_1^2-x_1x_3+x_3^2) \cdot (x_2^2-x_2x_3+x_3^2) \\
\cdot
& (x_1^3+x_2^3+x_3^3) \cdot (x_0^3+x_2^3+x_3^3) \cdot (x_0^3+x_1^3+x_3^3) \cdot (x_0^3+x_1^3+x_2^3).
\end{align*}
This is one of the few cases where symbolic computation of the equation of $\mathcal{P}(X)$ is easy.
\hfill $\diamondsuit$
\end{example}

\begin{example}\label{ex:stabbing_victory} \rm Let $d=3$ and fix  the cubic
$\,f \,=\, x_1^3 + x_2^3 + x_3^3 + x_4^3 + (x_1 + 2 x_2 + 3 x_3 + 4 x_4)^3$.
It defines our surface $X$.
Using the method above, we rapidly compute the ideal of $\mathcal{P}^\ell(X)$.
We demonstrate how to find the visual events of type $\mathcal{P}$ as the camera moves along a line. 

Consider the line with parametric representation $\, z(t) \,=\, (t: 1: t-1:
t+1)\,$ in $\PP^3$. Let $Q$ be the skew-symmetric $4 \times 4$ matrix obtained from
(\ref{eq:skewsymP}) by substituting to dual Pl\"ucker coordinates. 
We add the four coordinates of $ z(t) \cdot Q$ to the
ideal of $\mathcal{P}^\ell$, we then saturate with respect to $\langle
q_{12},\ldots,q_{34} \rangle$, and thereafter we eliminate the unknowns $q_{ij}$. The result is
$$ 
\begin{matrix}
495403946635821355157683145728 t^{30}+4349505253226024309192581220352 t^{29} \\ + 
18437739306679654261938338946432 t^{28}+50562321054013553614808463278912 t^{27} \\
+ \,\cdots\,\, \cdots\,\,
-81509153943200707008 t^2
-1885273424647073088 t -19650742648215232.
\end{matrix}
$$
This polynomial has $30$ distinct complex roots. Precisely $8$ of them are real, namely
$$ \begin{small} \begin{matrix} \!\!
\bigl\{ -1.01358602985259876, -1.011352289518814, -0.600974923580648806,
    -0.35014676100811994, \\ \quad -0.2668550692437184,\, -0.191676056625314,\,
    -0.0811161566932513655, \, 0.378943747730770221 \bigr\}.
\end{matrix}    
\end{small}
    $$
    These $8$ roots mark the visual events of type $\mathcal{P}$ 
    as the viewpoint travels along the line $z(t)$.    
    
    The univariate polynomial of degree $30$ can also be computed from the
    dual curve $\mathcal{P}^p$. Let $g_1$ and $g_2$ be  the polynomials in $y_1,y_2,y_3,y_4$ 
    of degree four and six promised in   Proposition  \ref{ref:quarticsextic}. We
    augment the ideal $\,I(\mathcal{P}^p) = \langle g_1,g_2 \rangle \,$
    by the $3 \times 3$-minors of the $3 \times 4$-matrix
    $$ \begin{pmatrix}
    \partial g_1 / \partial x_1 &  \partial g_1 / \partial x_2 &  \partial g_1 / \partial x_3 &  \partial g_1 / \partial x_4 \\
    \partial g_2 / \partial x_1 &  \partial g_2 / \partial x_2 &  \partial g_2 / \partial x_3 &  \partial g_2 / \partial x_4 \\
  t & 1 & t-1 & t+1 \end{pmatrix}. $$
  We then saturate the resulting ideal by the ideal of the six $2 \times 2$-minors in
  first two rows, and finally we eliminate $x_1,x_2,x_3,x_4$.
  This gives the same polynomial of degree $30$ in $t$.    
\hfill $\diamondsuit$
\end{example}

We found that the computation of the edge surface $\mathcal{E}(X)$ is more
challenging than that of the parabolic surface $\mathcal{P}(X)$. Consider the
case when $X$ is a general quartic. Here,  the surface $\mathcal{E}(X)$ has
degree $160$, and hence so does the curve $\mathcal{E}^\ell(X)$ in ${\rm
Gr}(1,\PP^3)$. We succeeded in computing the ideal of this curve only for
quartics $X$ that are singular or very special. For instance, if $X $ is the
Fermat quartic then $\mathcal{E}(X)$ a surface of degree $80$, with
multiplicity $2$. Since $\mathcal{E}(X)$ is developable, we could also try to
use $\mathcal{E}^p(X)$ as an encoding. Unfortunately, the degree is then even
higher. Namely, by Proposition~\ref{prop:dual_edge_parabolic}, the dual curve
$\mathcal{E}^p(X)$ has degree $480$ in $(\PP^3)^*$. The  computation of edge
surfaces $\mathcal{E}(X)$ definitely requires further research.

\bigskip \bigskip 

\begin{small}

\noindent
{\bf Acknowledgments.} This project started in May 2016 at the GOAL workshop
  in Paris. We thank Mohab Safey El Din, Jean-Charles Faug\`ere, Jon
  Hauenstein, and Jean Ponce for their help in the initial stages. We are also
  grateful to Joachim Rieger for an inspiring discussion on singularity
  theory, and to Emre Sert\"oz for helpful comments on intersection theory.
  Kathl\'en Kohn was funded by   the Einstein Foundation Berlin.  Bernd
  Sturmfels received partial support from the US National Science Foundation
  (DMS-1419018) and the Einstein Foundation Berlin. Matthew Trager was supported in part by the ERC advanced grant
  VideoWorld, the Institut Universitaire de France, the Inria-CMU
  associated team GAYA, and the ANR grant RECAP.
\end{small}

\bigskip \bigskip

\begin{small}

\end{small}

\bigskip \bigskip

\noindent
\footnotesize {\bf Authors' addresses:}

\smallskip

\noindent Kathl\'en Kohn,
TU Berlin, Germany,
{\tt kohn@math.tu-berlin.de}.

\smallskip

\noindent Bernd Sturmfels,
MPI Leipzig, Germany,
 and UC Berkeley, USA,
 {\tt bernd@mis.mpg.de},
{\tt bernd@berkeley.edu}.

\smallskip

\noindent Matthew Trager,
Inria, \'Ecole Normale Sup\'erieure Paris, CNRS, PSL Research University,
France, {\tt matthew.trager@inria.fr}.

\end{document}